\documentclass[12pt]{amsart}
\usepackage{amsthm}
\usepackage{amssymb}
\usepackage{amsmath}
\usepackage{hyperref}
\usepackage{xcolor}
\usepackage{amsthm}
\usepackage{dsfont}
\usepackage[matrix,arrow]{xy}

\usepackage{enumerate}

\usepackage[style=trad-abbrv,maxnames=99,maxalphanames=9, isbn=false, giveninits=true, doi=false, url=false]{biblatex}
\bibliography{ref.bib}
\renewbibmacro{in:}{}

\DeclareMathAlphabet\mathbfcal{OMS}{cmsy}{b}{n}

\numberwithin{equation}{section}

\newtheorem{theorem}{Theorem}[section]

\newtheorem{lemma}[theorem]{Lemma}

\newtheorem{proposition}[theorem]{Proposition}
\newtheorem{corollary}[theorem]{Corollary}
\newtheorem{remark}[theorem]{Remark}

\newtheorem{exam}[theorem]{Example}

\newtheoremstyle{named}{}{}{\itshape}{}{\bfseries}{.}{.5em}{\thmnote{#3\ }#1}
\theoremstyle{named}

 \newcommand{\RR}{\mathbb{R}}

\newcommand{\cE}{\mathcal{E}}
\newcommand{\cG}{\mathcal{G}}

\newcommand{\cC}{\mathcal{C}}

\newcommand{\cN}{\mathcal{N}}

\newcommand{\cV}{\mathcal{V}}
\newcommand{\cW}{\mathcal{W}}

\newcommand{\vol}{\operatorname{vol}}

\newcommand{\ddc}{\sqrt{-1}\partial\bar\partial}

\newcommand{\tr}{\mathrm{tr}}

\title{On relative $L^\infty$ estimate for complex Monge-Amp\`ere equations
}

\usepackage{hyperref}
\hypersetup{colorlinks,linkcolor={blue},citecolor={red}}

\begin{document}

\author{Junbang Liu}

\begin{abstract}
We prove a relative $L^\infty$ estimate for a class of complex Monge-Ampère type equations on K\"ahler manifolds. It provides a unified approach to Tundinger type estimate and uniform estimate. It also improves the previous results about modulus of continuity, stability estimates, and $W^{1,1}$-estimates of Green's functions. The argument is based on the PDE method developed by Guo-Phong-Tong and constructing appropriate comparison metrics from entropy bound.  
\end{abstract}
\maketitle
\tableofcontents

\section{Introduction}
The primary goal of this paper is to improve upon the results of Guo, Phong, and Tong regarding PDE methods for obtaining $L^\infty$-estimates for complex Monge-Ampère type equations. The complex Monge-Ampère equation has been extensively studied since Yau’s seminal work \cite{https://doi.org/10.1002/cpa.3160310304} in solving Calabi’s conjecture. Yau established uniform estimates for the potential using Moser’s iteration method, where his estimates relied on the $L^\infty$-norm of the volume form. A significant advancement was made by Kołodziej \cite{10.1007/BF02392879}, who applied pluripotential theory to bound the metric potential by an Orlicz norm of the volume form. The key ingredient in Kołodziej’s approach is the capacity introduced by Bedford and Taylor \cite{MR0674165}. An alternative approach, using the Alexandrov-Bakelman-Pucci (ABP) maximum principle, was developed by Błocki \cite{MR2156505} \cite{MR2817572}, where the bound depends on the $L^p$-norm of the volume form for some $p > 2$. More recently, a breakthrough was achieved by Guo, Phong, and Tong \cite{MR4593734,guo2022uniform}, based on a novel idea from Chen and Cheng \cite{CC1}, which compares the equation with an auxiliary complex Monge-Ampère equation. Their method is purely PDE-based, employing De Giorgi’s iteration method to establish $L^\infty$-estimates for a class of fully nonlinear equations. Notably, their method reduces the dependence of the $L^\infty$-norm of the potential to the $L^1(\log L)^p$-norm of the volume form and also makes the estimates uniform in cases where the background metric degenerates.

We consider the general Hessian equations on a compact K\"ahler manifold $(M^n,\omega_M)$:\begin{equation}\label{generalhessianequations}
\begin{aligned}
&g(\lambda[h_\varphi])=c_\omega e^{F},\\
&\lambda[h_\varphi]\in\Gamma, \sup\varphi=0.
\end{aligned}
\end{equation}
We also assume that the function $F$ is normalized as $\int_Me^{nF}\omega_M^n=\int_M\omega_M^n.$ Here $g(\lambda_1,\ldots,\lambda_n)$ is a symmetric function defined on an open convex symmetric cone $\Gamma$ with $\Gamma_n:=\{\lambda_i>0 \text{ for all }i=1,\ldots, n\}\subset\Gamma\subset\Gamma_1:=\{\sum_{i=1}^{n}\lambda_i>0\}$.
$\omega$ is another K\"ahler metric on $M$ and $\omega_\varphi:=\omega+\ddc\varphi$. We use $h_\varphi$ to denote the endomorphism between $TM$ and itself defined by $h_\varphi:=\omega_M^{-1}\omega_\varphi$. Let $\lambda[h_\varphi]$ denote the eigenvalues of $h_\varphi$. We put the following structure conditions on $g$:\begin{enumerate}
\item $\frac{\partial g}{\partial\lambda_i}>0 $ for any $i=1,\ldots,n$.
\item There exists a positive constant $\gamma_0>0$ such that $\det(\frac{\partial g}{\partial h_{ij}})\geq \gamma_0>0$.
\item There exists a positive constant $c_0$ such that $\sum_{i=1}^{n}\lambda_i\frac{\partial g}{\partial\lambda_i}\leq c_0g$.
\end{enumerate}

Remarkably, there is a large class of operators satisfying the structure conditions proved by Harvey-Lawson in \cite{MR4589710}.
Then, our main theorem states:
\begin{theorem}
     Assume $\omega\leq \kappa\omega_M$ for some constant $\kappa>0$, and $\varphi$ be a smooth solution to equation \ref{generalhessianequations}. For any increasing function $\Phi(x):\mathbb{R}\rightarrow \mathbb{R}^+$  with $\Phi(\infty)=\infty$, define a function $h(s)$ as $h(s)=\int_0^s\Phi(x)^{-\frac{1}{n}}dx$. Then, there exists a $\omega$-plurisubharmonic function $\psi$ with normalization $\sup\psi=0$, such that for any constant $\beta>0$, there exists constants $C_1,C_2,C_3, C_4, C_5$ depending on $M,n,\omega_M,\Phi, \frac{c_\omega^n}{V_\omega},\cN_\Phi(F),\beta,\kappa,c_0,\gamma_0$ s.t. \[
    -\varphi\leq C_1h(-\beta\psi+C_2)+C_3.
\]
where $V_\omega:=\int_M\omega^n, \cN_\Phi(F):=\int_M\Phi(F)e^{nF}\omega_M^n$. In particular, \begin{enumerate}
        \item If $\Phi(x)=|x|^p$ when $x>1$, $0<p<n$, then $\int_M\exp\{C_4(-\varphi)^{\frac{n}{n-p}}\}\omega^n\leq C_{5}.$
\item If $\Phi(x)=|x|^n$ when $x>1$, then $
    \int_M\exp\{\exp{C_4(-\varphi)}\}\omega^n\leq C_5.$
\item If $\Phi(x)=|x|^n(\log |x|)^p$, when $x>1$,$0<p<n$, $    \int_M\exp\{\exp\{C_4(-\varphi)^{\frac{n}{n-p}}\}\}\omega^n\leq C_5.$
\item If $\int_0^\infty \Phi(t)^{-\frac{1}{n}}dt<\infty$, then $
    -\varphi\leq C_4.$
\item If $\Phi(x)$ is just increasing to $\infty$, we have $
    \int_Me^{-\lambda\varphi}\omega^n\leq C $ for $\forall \lambda>0$ with $C$depending on the above mentioned parameters and  additionally on $\lambda.$ 
    \end{enumerate}
\end{theorem}

In this setting, the theorems of Guo, Phong, and Tong \cite{MR4593734} and Guo and Phong \cite{guo2022uniform} are special cases where $\Phi(x) = |x|^p$. To prove our result, we combine the idea of Darvas, Di Nezza, and Lu \cite{MR4211083} on relative $L^\infty$-estimates with the novel idea of Chen and Cheng \cite{CC1}, which involves bounding the solution using entropy. This combination leads to several improvements in the PDE method of Guo, Phong, and Tong and ultimately provides the result stated above. Specifically, we first employ the method of Guo, Phong, and Tong to establish a relative uniform estimate, which generalizes the result of Darvas, Di Nezza, and Lu \cite{DDL21log-concave}. We then incorporate the idea of Chen and Cheng, using entropy to bound the solution, to construct the necessary comparison Kähler metric and comparison function with a uniform $L^q$-norm for any fixed $q > 1$. This construction, together with the improved relative estimate, offers a unified approach to derive uniform Moser-Trudinger inequalities and $L^\infty$-estimates from uniform entropy bounds. It is worth noting that a crucial intermediate step in the proof of Guo, Phong, and Tong is to derive a uniform energy estimate using an entropy bound. These energy estimates, or the related Moser-Trudinger inequalities, were previously proven for the non-degenerate case in \cite{MR4593734} via the Alexandrov-Bakelman-Pucci (ABP) estimate, and later extended to the degenerate case by Guo and Phong \cite{guo2022uniform}. An interesting part in our proof is that we can bypass the need for these energy estimates by constructing a comparison function with a uniform $L^q$-norm for any fixed $q > 1$. 

An application of our method of relative $L^\infty$-estimate is to show the modulus of continuity of solutions to complex Monge-Ampère equations. It sharpens the results of Guo-Phong-Tong-Wang in \cite{guo2021modulus} in the sense that under their assumption of $L^1(\log L)^p$-bound of $e^F$, their results on the modulus of continuity $m(r)=\frac{C}{(\log dist_{\omega_M}(x,y))^{\alpha}}$ with $\alpha=\min\{\frac{p}{n}-1,\frac{p}{n+1}\}$ can actually be improved to $\alpha=\frac{p}{n}-1$. We also proved the modulus of continuity under assumptions on some weaker Orlicz norm. Note that such estimates are also proved by Gurdj, Guenancia and Zeriahi in \cite{GGZ} by pluripotential theory. 
\begin{theorem}Let $(M^n,\omega)$ be a compact K\"ahler manifold with a fixed K\"ahler metric $\omega$. Fix $p>n$, let $u$ be a smooth solution to the equation \[(\omega+\ddc u)^n=e^F\omega^n, \quad \omega+\ddc u>0.
\]There exist constants $C_1=C_1(n,p,\omega,||e^F||_{L^1(\log L)^p}),C_2=C_2(n,p,\omega,||e^F||_{L^1(\log L)^n(\log\log L)^p})$ such that \begin{equation}
|u(x)-u(y)|\leq \frac{C_1}{|\log d_{\omega}(x,y)|^{\frac{p}{n}-1}},
\end{equation}
and \begin{equation}
|u(x)-u(y)|\leq \frac{C_2}{\left(\log(-\log d_{\omega}(x,y))\right)^{\frac{p}{n}-1}}.
\end{equation}
\end{theorem}


One more application of the method of relative uniform estimate is to prove a stability estimate for complex Hessian equations. Such type of results were proved in \cite{MR1986892,MR3219505,MR2669357} by pluripotential theory and \cite{MR4505150} by their PDE method. Roughly speaking, the stability estimate says that when the right-hand sides of the equations are in some fixed Orlicz space and $L^1$-close, then the potentials are $L^\infty$-close. 
In the spirit of \cite{MR4505150}, we prove a stability results by assuming an almost weakest Orlicz-norm on the right-hand side (weak in the sense of weakest norm to ensure an $L^\infty$-bound of the potential). Our theorem states \begin{theorem}Let $(M^n,\omega_M)$ be a compact K\"ahler manifold with a given K\"ahler metric $\omega_M$ with normalization $[\omega_M]^n=1$. Let $\Phi$ be an increasing positive function such that $ \int^\infty\Phi^{-\frac{1}{n}}(t)dt<\infty$ and $\Phi(x)\leq e^{C_0x}$ for some $C_0>0$ when $x>1.$
Given $K>0$, let $G,H$ be smooth functions in\[
    \mathcal{K}:=\left\{F|\int_Me^F\omega_M^n=\int_M\omega_M^n, \int_M\Phi(|F|)e^{\frac{n}{k}F}\omega_M^n\leq K\right\}.
\]
Let $\chi$ be a closed nonnegative real $(1,1)$-form on $M$, and define $\omega_t=\chi+t\omega_M$. Assume that $u,v$ are smooth admissable solutions for \[
    (\omega_t+\ddc u)^k\wedge\omega_M^{n-k}=c_{t}e^G\omega_M^n,\quad (\omega_t+\ddc v)^k\wedge\omega_M^{n-k}=c_te^H\omega_M^n,
\]where $c_t:=\int_M\omega_t^k\wedge\omega_M^{n-k}$. We normalize $u,v$ as $    \max(u-v)=\max(v-u).$ Then there exists a constant $C(||e^G-e^H||_{L^1(\omega_M^n)}|M,\omega_M,n,K,,\Phi,\frac{c_t^{\frac{1}{k}}}{V_t^{\frac{1}{n}}}, C_0)$ such that \[
    \max(u-v)\leq C(||e^G-e^H||_{L^1(\omega_M^n)}|M,\omega_M,n,K,\Phi,,\frac{c_t^{\frac{1}{k}}}{V_t^{\frac{1}{n}}}, C_0),
\]
where $V_t:=\int_M\omega^n,$ and the constant $C(t|a,b,\ldots)$ means a constant depending on $t,a,b,\ldots$ and $\lim_{t\rightarrow 0}C(t|a,b,\ldots)=0.$
\end{theorem}
We remark that the constant $C(t|a,b,\cdots)$ can be made explicit(for the detailed expression, see section 5). Previously, the results of Guo, Phong, and Tong \cite{MR4505150} require that $e^F$ has uniform $L^p$ bound for some $p>\frac{n^2}{k}$ when the K\"ahler metric degenerates to a big semi-positive form. Together with the uniform energy estimates for degenerate metric proved by Guo and Phong in \cite{guo2022uniform}, it easy to see that their proof for nondegenerated metric works also and the requirement for $e^F$ is that it has uniform $L^p$ norm for some $p>\frac{n}{k}$. However, when one tries to prove such stability results when $e^F$ belongs to the space defined in our theorem which is just enough to ensure the uniform $L^\infty$-estimate, the main difficulty is to estimate the volume of the level set (see lemma \ref{controlAs}). It's interesting to find that our relative method can be applied and the difficulty mentioned above was overcome by noting that we can choose the volume form to be uniformly bounded in $L^q$ for any fixed $q>1$. 

Another application of our method of relative uniform estimate is to uniformly bound the diameter of Kähler manifolds with Kähler metrics belonging to the subset $\cW$ (defined in Section 7) of the Kähler cone. There are a wealth of research on diameter bounds for Kähler manifolds (for example, see \cite{MR1807951, MR2538503, MR2863918, guo2022diameter, guo2021modulus, MR4255068, GGZ}, and references therein). Notably, two works \cite{guo2022diameter, MR4255068} have introduced new insights into the diameter bounds of Kähler metrics. In \cite{MR4255068}, Li established a connection between the modulus of continuity of potentials and the diameter bound, and further advancements in this direction were made in \cite{guo2021modulus, GGZ}. Meanwhile, in \cite{guo2022diameter, guo2024diameter2}, Guo, Phong, Song, and Sturm derived a diameter bound by obtaining uniform $W^{1,1}$-estimates for the Green’s function. Following this approach, we prove estimates on the Green’s functions for metrics in class $\cW$ with some weaker bound on entropy, and derive similar diameter bounds. We remark that a key advantage of the pluripotential method is its direct extension to singular Kähler metrics. On the other hand, Guo, Phong, Song, and Sturm, using their PDE-based method, also proved a diameter bound in \cite{guo2023sobolev} by extending the Sobolev inequality to singular spaces. The Sobolev inequality implies non-collapsing results, which, in turn, lead to diameter bounds in singular spaces.

\begin{theorem}\label{greenfunction}
For any given $p>2n$,let $\Phi(x)=|x|^n(\log |x|)^{p}$.
There exists a constant $C=C(M,\omega_M,n,A,K,p)$, such that for any $\omega\in \mathcal{W}$, we have \begin{enumerate}
\item $ -\inf G(x,\cdot)[\omega]^n+\int_M|G(x,\cdot)|\omega^n+\int_M|\nabla G(x,\cdot)|\omega^n\leq C$
\item $diam(M,\omega)\leq C$
\end{enumerate}
\end{theorem}
The paper is organized as follows. In Section 2, we discuss the improved relative estimate and provide an application to cusp metrics. Section 3 focuses on the construction of comparison Kähler metrics and comparison functions with uniform $L^q$ bounds for any given $q > 1$, along with an application to $L^\infty$-estimates for the cscK equation with an integral bound on the scalar curvature. In Section 4, we explore the application of the method to the modulus of continuity of solutions to complex Monge-Ampère equations. In section 5 we prove the stability estimates. The generalization of the uniform estimate to the nef class is discussed in Section 6. In Section 7, we prove geometric estimates on Green’s functions and establish diameter bounds. Finally, in Section 8, we deal with the relative estimates for general equations.

\textbf{Acknowledgements}. The author would like to thank his advisor, Professor Xiuxiong Chen, for his guidance and encouragement. He also thanks Jingrui Cheng for reading the draft and helpful comments and suggestions. The author was partially supported by Simons Foundation International, LTD.

\section{Relative \texorpdfstring{$L^\infty$}{l}-estimate}
\begin{proposition}\label{babyrelativeestimate}Let $(M^n,\omega)$ be a compact K\"ahler manifold. Assume $\omega_\varphi,\omega_\psi$ are two smooth K\"ahler metrics satisfying \[
    \omega_\varphi^n\leq a^n\omega_\psi^n+f\omega^n,
\]
for some $a\in [0,1)$ and smooth positive function $f$. We normalize the K\"ahler potentials as $\sup\varphi=\sup\psi=0$. Given $q>1$, there exists a constant $C=C(M,\omega,n,||f||_{L^q(\omega^n)},(1-a)^{-n-1},q)$ such that \[
    \varphi\geq a\psi-C.
\]
\end{proposition}

\begin{proof}Let $\psi_1$ be the solution to the auxiliary equation \[
    (\omega+\ddc\psi_{1,j})^n=\frac{\eta_j(-\varphi+a\psi-s)f\omega^n}{A_{s,j}}, \quad \sup\psi_{1,j}=0.
\]
Here $\eta_j(x)=\frac{1}{2}(x+\sqrt{x^2+\frac{1}{j^2}})$, which is a smooth approximation of $x_+$ from above as $j\rightarrow \infty$, and \[
    A_{s,j}=\frac{1}{[\omega]^n}\int_M\eta_j(-\varphi+a\psi-s)f\omega^n.
\]
Note that $A_{s,j}\rightarrow A_s:=\frac{1}{[\omega]^n}\int_M(-\varphi+a\psi-s)_+f\omega^n$ as $j\rightarrow \infty$. 

Consider the function $\Psi_j:=(-\varphi+a\psi-s)-\epsilon_{1,j}(-\psi_{1,j}+\Lambda_{1,j})^{\frac{n}{n+1}}$. Choose $\epsilon_{1,j},\Lambda_{1,j}$ to be \[
\epsilon_{1,j}=\left(\frac{n+1}{n}\right)^{\frac{n}{n+1}}A_{s,j}^{\frac{1}{n+1}}, \quad \Lambda_{1,j}=\frac{n}{n+1}\frac{1}{(1-a)^{n+1}}A_{s,j}.
\]
We claim that $\Psi_j\leq 0$. Assume that $\Psi_j$ attains its maximum at $x_0$, without loss of generality we can assume $x_0\in \Omega_s:=\{-\varphi+a\psi-s>0\}$. 
Then at $x_0$, we have (for convenience, we suppress the subscript $1,j$ in $\epsilon_{1,j},\Lambda_{1,j}$)\[
    \begin{aligned}
    0\geq &-\ddc\varphi+a\ddc\psi+\epsilon\frac{n}{n+1}(-\psi_1+\Lambda)^{-\frac{1}{n+1}}\ddc\psi_1\\
    &+\epsilon\frac{n}{(n+1)^2}\sqrt{-1}\partial{\psi_1}\wedge\bar{\partial}\psi_1\\
    \geq & -\omega_\varphi+a\omega_\psi+\epsilon\frac{n}{n+1}(-\psi_1+\Lambda)^{-\frac{1}{n+1}}\omega_{\psi_1}.
\end{aligned}
\]
For the last inequality above we used that $(1-a)= \epsilon\frac{n}{n+1}\Lambda^{-\frac{1}{n+1}}$.
Therefore\[
\begin{aligned}
a^n\omega_{\psi}^n+f\omega^n\geq \omega_\varphi^n&\geq (a\omega_{\psi}+\epsilon\frac{n}{n+1}(-\psi_1+\Lambda)^{-\frac{1}{n+1}}\omega_{\psi_1})^n\\
&\geq a^n\omega_\psi^n+(\frac{\epsilon n}{n+1})^{n}(-\psi_1+\Lambda)^{-\frac{n}{n+1}}\eta_j(-\varphi+a\psi-s)f\omega^n\frac{1}{A_{s,j}}.
\end{aligned}
\]
Note that $\eta_j(x)\geq x_+$, we have \[
    1\geq (\frac{\epsilon n}{n+1})^{n}(-\psi_1+\Lambda)^{-\frac{n}{n+1}}(-\varphi+a\psi-s)_+\frac{1}{A_{s,j}}.
\]
With the choice of $\epsilon_1,\Lambda_1$, it is equivalent to $\Psi(x_0)\leq 0$

Next, by the H\"older inequality, we have \[
    [\omega]^nA_s\leq \int_M(-\varphi+a\psi)_+f\omega^n\leq (||\varphi||_{L^{\frac{q}{q-1}}(\omega^n)}+||\psi||_{L^{\frac{q}{q-1}}(\omega^n)})||f||_{L^q(\omega^n)}\leq C(\alpha_{M,\omega},q)||f||_{L^q(\omega^n)},
\]where we bound the $L^{\frac{q}{q-1}}$-norm of $\varphi,\psi$ uniformly by Tian's $\alpha$-invariant $\alpha_{M,\omega}$. Since $A_{s,j}$ converges to $A_s$ as $j\rightarrow \infty$, we get\[
    \Lambda_{1,j}=\frac{n}{n+1}\frac{1}{(1-a)^{n+1}}A_{s,j}\leq C(\frac{1}{[\omega^n]},n,(1-a)^{-n-1},\alpha_{M,\omega},q)||f||_{L^q(\omega^n)}\quad \text{ for } j \text{ large}. 
\]
Denote $\phi(s):=\frac{1}{[\omega^n]}\int_{\Omega_s}f\omega^n.$ Then, on $\Omega_{s+t}$, we have $-\varphi+a\psi-s-t\geq 0$, which implies \[
    \frac{-\varphi+a\psi-s}{t}\geq 1.
\]
So we have \begin{equation}\label{tphisplust}
t\phi(s+t)\leq \frac{1}{[\omega]^n}\int_{\Omega_s}(-\varphi+a\psi-s)_+f\omega^n=A_s.
\end{equation}
On the other hand, since $\Psi_j\leq 0$, we have\begin{equation}\label{estimateAs}
\begin{aligned}
 A_s\leq &\frac{1}{[\omega]^n}\int_{\Omega_s}\epsilon_{1,j}(-\psi_{1,j}+\Lambda_{1,j})^{\frac{n}{n+1}}f\omega^n\\
 =&C(M,\omega,n)A_{s,j}^{\frac{1}{n+1}}\int_{\Omega_s}(-\psi_{1,j}+\Lambda_{1,j})^{\frac{n}{n+1}}f\omega^n\\
 \leq &C(M,\omega,n)A_{s,j}^{\frac{1}{n+1}}\left(\int_{\Omega_s}(-\psi_{1,j}+\Lambda_{1,j})^{\frac{np}{n+1}}f\omega^n\right)^{\frac{1}{p}}\left(\int_{\Omega_s}f\omega^n\right)^{1-\frac{1}{p}}\\
 \leq &C(M,\omega,n)A_{s,j}^{\frac{1}{n+1}}||(-\psi_{1,j}+\Lambda_{1,j})^{\frac{np}{n+1}}||_{L^{q^*}(\omega^n)}^{\frac{1}{p}}||f||_{L^q(\omega^n)}^{\frac{1}{p}}\phi(s)^{1-\frac{1}{p}}\quad  \quad (q^*=\frac{q}{q-1})\\
 \leq &C(M,\omega,n)A_{s,j}^{\frac{1}{n+1}}(||(-\psi_{1,j})||_{L^{\frac{npq^*}{n+1}}(\omega^n)}+\Lambda_{1,j})^{\frac{n}{n+1}}||f||_{L^q(\omega^n)}^{\frac{1}{p}}\phi(s)^{1-\frac{1}{p}}
\end{aligned}
\end{equation}
Let $j\rightarrow \infty$, we get \[
    A_s\leq c(n,M,\omega,p,q,(1-a)^{-n-1})A_s^{\frac{1}{n+1}}(1+||f||_{L^{q}(\omega^n)})^{\frac{n}{n+1}}||f||_{L^{q}(\omega^n)}^{\frac{1}{p}}\phi(s)^{1-\frac{1}{p}}.
\]
In the above inequality, we have used Tian's $\alpha$-invariant to bound the integral $\int_M(-\psi_{1,j})^{\frac{npq^*}{n+1}}\omega^n$ again.

Choose $p>n+1$, and let $\delta_0=\frac{n+1}{n}(1-\frac{1}{p})-1>0$.

Combine (\ref{tphisplust}) and (\ref{estimateAs}) we get $t\phi(t+s)\leq c(n,M,\omega,p,q,(1-a)^{-n-1}, ||f||_{L^q(\omega^n)})\phi(s)^{1+\delta_0}.$
Now we can use the following De Giorgi's iteration lemma to close the proof.
\begin{lemma}[De Giorgi]\label{DeGiorgiiterationlemma}Let $\phi:\RR\rightarrow \RR$ be a monotone decreasing function such for some $\delta_0>0$ and any $s\geq s_0,t>0,$\[
   t\phi(t+s)\leq C_0\phi(s)^{1+\delta_0}. 
\]
Then $\phi(s)=0$ for any $s\geq \frac{2C_0\phi^{\delta_0}(s_0)}{1-2^{-\delta_0}}+s_0$.
\end{lemma}
\end{proof}
{}
\begin{remark}As in \cite{MR4593734}, we can also just assume the $L^1(\log L)^p$-bound for $f$ and apply the Young inequality in the iteration process to get the uniform estimate. For our applications, the $L^p$-bound is enough since we are able to construct such comparison K\"ahler metric  $\omega_\psi$ and function $f$ with uniform $L^q$ norm for arbitrary large fixed $q$.
\end{remark}

As an application of the idea of relative estimates,  we show a relative uniform estimate for the complex Monge-Ampère equation on quasi-projective manifolds. Let $(M^n,\omega)$ be a compact K\"ahler manifold with a K\"ahler metric $\omega$. Let $D$ be a smooth divisor on $M$ and we use $L_D$ to denote the holomorphic line bundle defined by $D$. Fix a hermitian metric $h$ on $L_D$ and choose a defining section $s_D$ of $L_D$, i.e. $\{s_D=0\}=D$, with $||s_{D}||_{h}^2\leq 1$. Let $f$ be a smooth positive function on $M$. We consider the following Monge-Ampère equation on $M$:\begin{equation}\label{MAeq}
(\omega+\ddc\varphi)^n=\frac{f\omega^n}{||s||_h^2\log^2(||s||_h^2)}, \quad \sup \varphi=0.
\end{equation}

It's natural to look at the a priori estimate of solutions to equation \eqref{MAeq}  when study the problem of finding canonical metrics with cusp singularity, see \cite{MR3589348,MR3652249} and the reference therein. Our main result is the following\begin{theorem}\label{improvedboundtheorem}
Assume that $f$ is smooth function $M$ with $0<c_1\leq f\leq c_2$. There there exist constants $a_1,a_2,b_1,b_2,c_3,c_4$ depending on $M,\omega, n, c_1,c_2$ such that the unique solution $\varphi\in\cE(X,\omega)$ is smooth in $X\setminus D$, and moreover, \begin{equation}\label{doublebound}
    -a_1\log(b_1-\log(||s||_h^2))-c_1\leq\varphi\leq -a_2\log(b_2-\log(||s||_h^2))+c_2.
\end{equation}
\end{theorem}

The smoothness part and the relative lower bound of $\varphi$ is already proved in \cite{MR3652249} where the authors used the pluripotential theory and generalized  the concept of capacity to a relative setting to show it. They also give a relative upper bound (proposition 4.5 in \cite{MR3652249}) of $\varphi$ in the form of $-c(\log(-\log||s||_h^2))^{p}+C$, for any $p\in (0,1)$, but we improves it to the present form \eqref{doublebound} that $p$ can be $1$. 

Our strategy is as follows: firstly, we approximate the right hand side of equation by smooth functions, and show a rough apriori bound of $\varphi$. Such rough bound are already proved in \cite{MR3652249} with pluripotential theory but to be readers' convenience we use the PDE-method to show it again here.  The rough bound is enough for us to establish higher order estimates away from the divisor and pass it to the limit while doing approximation(at least on any compact subset of $M\setminus D$). Then, with such rough bounds, and regularity of solution on $M\setminus D$, the PDE-method can be applied to improve it to the present stated in the theorem by a relative comparison idea. 

\subsection{Rough bound}
In this section, we consider a more general class of complex Monge-Ampère equations on $M$,\begin{equation}\label{pshMA}
(\omega+\ddc\varphi)^n=ge^{-\phi},\quad \sup \varphi=0.
\end{equation} 
We assume that $\phi\in Psh(M,B\omega)\cap C^{\infty}$ for some constant $B$, $\phi\leq 0$ and $g$ is smooth. Fix an increasing function $\Phi(x):\mathbb{R}\rightarrow \mathbb{R}^+$ such that $\Phi(x)\rightarrow \infty$. The entropy of $ge^{-\phi}$ with respect to $\Phi$ is defined by $
    \cN_\Phi(ge^{-\phi})=\int_{M}ge^{-\phi}\Phi(\log(ge^{-\phi}))\omega^n.$
\begin{lemma}\label{roughuniformbound}Let $\varphi$ be a smooth solution to equation \eqref{pshMA}, for any $a\in(0,\frac{1}{B}), p>1$, there exist constant $C$ depending on $M,\omega,B,n,||g||_{L^p(\omega^n)},\cN_{\Phi}(ge^{-\phi}), p, a$, such that \[
    -\varphi+a\phi\leq C.
\]
\end{lemma}
\begin{proof}
The proof is similar as before. Without loss of generality, we normalize $\omega$ by $\int_M\omega^n=1.$ 
Let $\eta_j(x)=x+\sqrt{x^2+j^{-2}}$ which is a smooth approximation of the positive part $x_+$ of $x$. For any $s>0$, consider the auxiliary complex Monge-Ampère equation \[
    (\omega+\ddc\psi_{j})^n=\frac{\eta_{j}(-\varphi+a\phi-s)ge^{-\phi}\omega^n}{A_{s,j}},\quad \sup \psi_{j}=0,\]
where $A_{s,j}:=\int_M\eta_{j}(-\varphi+a\phi-s)g\omega^n$. 
Define the function $\Psi$ as \[
    \Psi:=(-\varphi+a\phi-s)-\epsilon_{j}(-\psi_{j}+\Lambda_{j})^{\frac{n}{n+1}},\]where \[\epsilon_j=\left(\frac{n+1}{n}\right)^{\frac{n}{n+1}}A_{s,j}^{\frac{1}{n+1}}, \quad \Lambda_j=\frac{n}{n+1}\frac{1}{(1-aB)^{n+1}}A_{s,j}.\]

    Then a similar argument by maximum principle as before together with the condition $B\omega+\ddc\phi\geq 0$ shows that $\Psi_j\leq 0.$ 
Now, we look at the quantity $A_{s,j}$, and claim that $A_{s,j}$ is uniformly bounded. The key observation is that on $\Omega_s$, we have $-a\phi\leq -\varphi-s$. By proposition 4.1 in \cite{MR4028264} or the Moser-Trudinger inequality proved in the next section, for any $\lambda>0$, we have $\int_{\Omega_s}e^{-\lambda\phi}\omega^n\leq C$ for some uniform constant $C$ which depends on $M,n,\omega,\Phi,\lambda, \cN_\Phi(ge^{-\phi}).$
Therefore, for large $j$, \[
\begin{aligned}
 A_{s,j}\leq & 2\int_{\Omega_s}(-\varphi+a\phi-s)ge^{-\phi}\omega^n\\
 \leq &2\int_{\Omega_s}(-\varphi+a\phi)ge^{-\frac{1}{a}\varphi}\omega^n\\
 \leq& C||g||_{L^p}||\varphi||_{L^{p_2}}||e^{-\frac{1}{a}\varphi}||_{L^{p_3}}\leq C,
\end{aligned}
\]for some uniform constant $C$, where $\frac{1}{p}+\frac{1}{p_2}+\frac{1}{p_3}=1.$
Hence the constant $\Lambda_{j}$ is uniformly bounded.  And the iteration process is the same as before and using the observation again that on $\Omega_s$, we have $-a\phi\leq -\varphi-s$. We omit the details.
\end{proof}

\begin{lemma}Assume in addition that $\ddc\log g\geq -D\omega$, then the solution to equation \eqref{pshMA} satisfies \[
    n+\Delta\varphi\leq Ce^{-\phi},
\]
For some constant $C$ depending on $M,n,\omega,B,D,||g||_{L^p(\omega^n)},\cN_{\Phi}(ge^{-\phi})$.
\end{lemma}
\begin{proof}Standard computation shows that \[
    \Delta_{\varphi}\log(n+\Delta\varphi)\geq \frac{\Delta\log g-\Delta\phi}{n+\Delta\varphi}-C\tr_{\varphi}g,
\]where $C$ is a constant depends on the lower bound of bisectional curvature of $\omega$.
By our assumption that $\Delta\log g\geq -D\omega$ and $\Delta\phi\geq -B\omega$, we have \[
\begin{aligned}
\Delta_\varphi\log(n+\Delta\varphi)\geq& -\frac{nD}{n+\Delta\varphi}-\frac{\Delta\phi}{n+\Delta\varphi}-C\tr_\varphi g\\
\geq &-(n^3D+C)\tr_\varphi g-\frac{\Delta\phi}{n+\Delta\varphi},
\end{aligned}
\]
where we have used the elementary inequality$(n+\Delta\varphi)\tr_{\varphi}g\geq n^2$ in the second line.     
Then we consider the function $u:=\log(n+\Delta\varphi)-\lambda\varphi+\phi$, then \[
\begin{aligned}
    \Delta_\varphi u\geq &(\lambda-n^3D-C-B)\tr_\varphi g-\frac{\tr_g (B\omega+\ddc\phi)}{n+\Delta\varphi}+\frac{Bn}{n+\Delta\varphi}-\lambda n+\tr_\varphi(B\omega+\ddc\phi)\\
    \geq &(\lambda-n^3D-C-B)\tr_\varphi g-\frac{\tr_g (B\omega+\ddc\phi)}{n+\Delta\varphi}-\lambda n+\frac{\tr_g(B\omega+\ddc\phi)}{n+\Delta\phi}
    \end{aligned}
\]
where we have used the inequality $\tr_\varphi\alpha\geq \frac{\tr_g\alpha}{n+\Delta\varphi}$ for any positive $(1,1)$-form $\alpha$. 
Choose $\lambda=n^3D+C+B+1$, we get \[
    \Delta_\varphi u\geq \tr_\varphi g-\lambda n.
\]So at the maximum of $u$, say $u(x_0)$, we have $\tr_\varphi g\leq C$ for some uniform constant $C$. Now by the elementary inequality 
\[
    n\left(\frac{\omega_\varphi^n}{\omega^n}\right)^{\frac{1}{n}}\leq (n+\Delta\varphi)\leq(\tr_\varphi g)^{n-1}\frac{\omega_\varphi^n}{\omega^n} ,
\]
We have $n+\Delta\varphi\leq C$ for some uniform constant $C$, and therefore \[
    u(x)\leq u(x_0)\leq C-\lambda\varphi(x_0)+\phi(x_0).
\]
So \[
    \log(n+\Delta\varphi)(x)\leq C+\lambda\varphi(x)-\lambda\varphi(x_0)-\phi(x)+\phi(x_0).
\]
By lemma \ref{roughuniformbound}, we have $-\varphi(x)+\frac{1}{\lambda}\phi(x)\leq C$ for some uniform constant $C$. Hence we arrive at \[
    \log(n+\Delta \varphi)\leq C+\lambda\varphi-\phi,
\]for some uniform constant $C$. 
\end{proof}
Let $\phi=-\log||s||_h^2-\log\log^2||s||_h^2$, straightforward computation shows that there exists a uniform constant $B$ depending on the curvature of $(L,h)$ such that $B\omega+\ddc\phi\geq 0.$ Approximate $\phi$ by smooth $B\omega$-psh functions shows  
\begin{proposition}The unique solution in $\cE(M,\omega)$ to equation (\ref{MAeq}) is indeed smooth in $M\setminus D$ and for any $a>0$, \[
 -\varphi+a\phi\leq C, \quad n+\Delta\varphi\leq Ce^{\lambda\varphi-\phi}
 \]
 \end{proposition}
\subsection{Improved bound}
In this section, we give the proof of theorem \ref{improvedboundtheorem}.

For $i=1,2$ , we let $\psi_i=-A_i\log(\lambda_i-\log||s||_h^2)$, where $A_i, \lambda_i$ are  constants to be determined. Then direct calculation shows that \begin{lemma}[Auvray\cite{MR3589348}]\begin{enumerate}
    \item Fix $A_i>0$, for sufficient large $\lambda_i$ depending on $A_i,\omega$, the form $\omega+\ddc\psi_i$ defines a K\"ahler metric on $M\setminus D$.
    \item Locally, if $D\cap U=\{z_1=0\}$ for $U$ a neighborhood of some points on $D$, we have the asymptotic expansion $\omega+\ddc\psi_i=\frac{A_idz_1\wedge d\overline{z_1}}{|z_1|^2\log^2(|z_1|^2)}+\omega|_D+O(\frac{1}{|\log(|z_1|^2)|})$
\end{enumerate}
\end{lemma}

As a consequence for any $a\in (0,1)$, we have the following comparing inequality on  $M\setminus D$ for suitable chosen $A_i,\lambda_i$,\begin{equation}\label{comparisonineq}
    \omega_{\psi_1}^n\leq a^n\frac{f\omega^n}{||s||^2\log^2(||s||^2)}+C_{A_1,\lambda_1,c_1,a,M,D}\omega^n,\end{equation}
    and \[
    \frac{f\omega^n}{||s||^2\log^2(||s||^2)}\leq a^n\omega^{n}_{\psi_2}+C_{A_2,\lambda_2,c_2,a,M,D}\omega^n.\]
where we have used the assumption $c_1\leq f\leq c_2$. Now we can prove the improved relative estimate. 
\begin{proof}[ Proof of theorem \ref{improvedboundtheorem}]
Let $\eta_k(x)$ be a sequence of smooth positive functions, approximating $x_+$ from above. Consider the auxiliary equation \[
    (\omega+\ddc\psi_{s,k})^n=\frac{\eta_{k}(-\psi_1+a\varphi+\epsilon\log||s||^2-s)\omega^n}{A_{s,k}},\quad \sup \psi_{s,k}=0.
\]
Define function $\Psi$ as \[
    \Psi(x)=-\psi_1+a\varphi+\epsilon\log||s||^2-s-\epsilon_{k,s}(-\psi_{s,k}+\Lambda_{k,s})^{\frac{n}{n+1}},
\]where \[
    \epsilon_{k,s}=\left(\frac{n+1}{n}\right)^{\frac{n}{n+1}}C_{A_1,\lambda_1,c_1,a,M,D}^{\frac{1}{n+1}}A_{s,k}^{\frac{1}{n+1}},\quad \Lambda_{k,s}=\frac{2^{n+1}}{(1-a)^{n+1}}\frac{n}{n+1}C_{A_1,\lambda_1,c_1,a,M,D}A_{s,k}.
\]
We claim that $\Psi\leq 0 $. Since $-\psi_1+\epsilon\log||s||^2\rightarrow -\infty$ when $x\rightarrow D$ for any $\epsilon>0$. We can assume that $x_0\notin D$ (in fact this holds for all $k$) and $x_0\in \Omega_{s}:=\{-\psi_1+a\varphi+\epsilon\log||s||^2-s>0\}$. At $x_0$, we have \[
    \begin{aligned}
    0\geq &-\ddc\psi_1+a\varphi+\epsilon\ddc\log||s||^2+\epsilon_{k,s}\frac{n}{n+1}(-\psi_{s,k}+\Lambda_{k,s})^{-\frac{1}{n+1}}\ddc\psi_{s,k}\\
    \geq &-\omega_{\psi_1}+a\omega_\varphi-\epsilon c(D,h)+(1-a-\epsilon_{k,s}\frac{n}{n+1}(\Lambda_{k,s})^{-\frac{1}{n+1}})\omega\\
    &+\epsilon_{k,s}\frac{n}{n+1}(-\psi_{s,k}+\Lambda_{k,s})^{-\frac{1}{n+1}}\omega_{\psi_{s,k}}
    \end{aligned}
\]In the second line, $c(L,h)$ is the curvature of the line budnle $L$ with hermitian metric $h$, which is a smooth form on $M$. So  if we take $\epsilon\leq \frac{1-a}{2||c(L,h)||_\omega}$, we have \[
    -\epsilon c(L,h)+(1-a-\epsilon_{k,s}\frac{n}{n+1}(\Lambda_{k,s})^{-\frac{1}{n+1}})\omega\geq 0.
\]
So \[
    \omega_{\psi_1}\geq a\omega_{\varphi}+\epsilon_{k,s}\frac{n}{n+1}(-\psi_{s,k}+\Lambda_{k,s})^{-\frac{1}{n+1}}\omega_{\psi_{s,k}}.
\]
Tanking $n$-th power and combine with the comparison inequality \eqref{comparisonineq}, we get \[
    a^n\omega_\varphi^n+C_{A_1,c_1,a,M,D}\omega^n\geq\omega_{\psi_{1}}^n\geq a^n\omega_{\varphi}^n+\left(\epsilon_{k,s}\frac{n}{n+1}\right)^n(-\psi_{s,j,k}+\Lambda_{k,s})^{-\frac{n}{n+1}}\omega_{\psi_{s,j,k}}^n.
\]
With our choice of $\epsilon_{k,s},\Lambda_{k,s}$, we have \begin{equation}\label{psiless0}
    \eta_{k}(-\psi_1+a\varphi+\epsilon\log||s||^2-s)\leq \epsilon_{k,s}(-\psi_{s,k}+\Lambda_{k,s})^{\frac{n}{n+1}}. 
\end{equation}
So from $\eta_k(x)\geq x_+$, we get the desired $\Psi\leq 0.$
The iteration process is similar as previous, and we get \[
-\psi_1+a\varphi+\epsilon\log||s||^2\leq C,
\] for some uniform $C$ which is independent on $\epsilon$. Let $\epsilon\rightarrow 0,$ we get one side inequality of theorem \ref{improvedboundtheorem}. The proof for the other side is similar. But this time we need use lemma \ref{roughuniformbound} to say that the maximum of $\Psi_2:=-\varphi+a\psi_2+\epsilon\log||s||^2-s-\epsilon_{2,k,s}(-\psi_{k,s}+\Lambda_{2,k,s})^{\frac{n}{n+1}}$ is not attained on $D$.

\end{proof}

\section{A unified approach to the Trudinger-type estimate and \texorpdfstring{$L^\infty$}{linfty}-estimate}
In this section, we use the relative $L^\infty$-estimate to prove the following Moser-Trudinger type estimate
 \begin{theorem}\label{mosertrudingertype}Let $(M^n,\omega)$ be a smooth compact K\"ahler manifold
 and $\omega_\varphi$ be a smooth K\"ahler metric such that $\omega_\varphi^n=e^F\omega^n$ for some smooth function $F$ with $\int_Me^F\omega^n=\int_M\omega^n$. Let $\Phi(x)$ be an increasing function mapping $\mathbb{R}$ to $\mathbb{R}^+$. Then there exists a constant $C_1,C_2>0$ depending on $M,n, \omega,\int_M\Phi(F)e^F\omega^n,\Phi$ such that 

\begin{enumerate}
\item If $\Phi(x)=x^p$, when $x>2$, $0<p<n$, \[
    \int_M\exp\{C_1(-\varphi)^{\frac{n}{n-p}}\}\omega^n\leq C_2.
\]
\item If $\Phi(x)=x^n$ when $x>2$,\[
    \int_M\exp\{\exp{C_1(-\varphi)}\}\omega^n\leq C_2.\]
\item If $\Phi(x)=x^n(\log x)^p$ when $x>2$, $0<p<n$, \[
    \int_M\exp\{\exp\{C_1(-\varphi)^{\frac{n}{n-p}}\}\}\omega^n\leq C_2.
\]
\item If $\int_1^\infty\Phi^{-\frac{1}{n}}(x)dx<\infty,$ \[
    -\varphi\leq C_2.
\]
\item If $\Phi(x)$ is just increasing to $\infty$, we have \[
    \int_Me^{-\lambda\varphi}\omega^n\leq C_3 \quad \forall \lambda>0, \text{ with } C_3 \text{ additionally depends on }\lambda.
\]
\end{enumerate}
\end{theorem}
\begin{remark}(1) is proved in \cite{MR4593734} by ABP-estimate and lately generalized to degenerated class in \cite{guo2022uniform}. (2) and (3) are generalizations of \cite{guo2022uniform}, and $(5)$ is also proved in \cite{MR4028264}.
\end{remark}

\begin{proof}Let's consider a function $h(s):\mathbb{R}^+\rightarrow \mathbb{R}$, which is increasing, concave. We use $\cN_\Phi(F)$ to denote the integral $\int_Me^F\Phi(F)\omega^n$. For simplicity, we normalize the volume of $\omega$ to be 1 in the following. Let $\psi_1$ be the solution to \begin{equation}\label{defpsi1}
\omega_{\psi_1}^n=\frac{\Phi(F)e^F\omega^n}{\cN_\Phi(F)}, \quad \sup\psi_1=0.
\end{equation}

In this section we use $\alpha$ to denote the $\alpha$-invariant with respect to $(M,\omega)$, and for any $q>1$ we define $\psi=-h(-\frac{\alpha}{q}\psi_1+\Lambda_1)$, then 
\[\ddc\psi=\alpha q^{-1}h'\ddc\psi+(-h'')\alpha^2q^{-2}\sqrt{-1}\partial\psi_1\wedge\bar{\partial}\psi_1\geq \alpha q^{-1}h'\ddc\psi_1,\]
since $h$ is concave. 
Hence, \[\omega+\ddc\psi\geq \frac{\alpha}{q}h'\omega_{\psi_1} \text{ when }\frac{\alpha}{q}h'(-\frac{\alpha}{q}\psi_1+\Lambda_1)\leq 1.\] 
Combining with equation \eqref{defpsi1}, we get $$\omega_\varphi^n=e^F\omega^n\leq \alpha^{-n}q^{n}h^{'-n}\frac{\cN_\Phi(F)}{\Phi(F)}\omega_{\psi}^n $$.

Now, for any $a\in(0,1)$, there are two cases, \begin{enumerate}
\item $a^n\alpha^nq^{-n}h^{'n}\Phi(F)\cN_\Phi(F)^{-1}\geq 1$;
\item $a^n\alpha^nq^{-n}h^{'n}\Phi(F)\cN_\Phi(F)^{-1}< 1$.
\end{enumerate}
For the first case, we have \begin{equation}
\omega_\varphi^n\leq a^n\omega_\psi^n.
\end{equation}
For the second case, we have \begin{equation*}\label{Fbound}
\Phi(F)\leq a^{-n}\alpha^{-n}q^{n}\cN_\Phi(F) h^{'-n}(-\frac{\alpha}{q}\psi_1+\Lambda_1).
\end{equation*}

Now we define $h$ as \begin{equation}\label{definitionofh}
    h(s)=\int_0^sa^{-1}\alpha^{-1}q\cN_\Phi(F)^{\frac{1}{n}}\Phi(t)^{-\frac{1}{n}}dt.
\end{equation}
Note that $h$ is defined to satisfy the equation \[
    \Phi(s)=a^{-n}\alpha^{-n}q^{n}\cN_\Phi(F) h^{'-n}(s).
\]
Since $\Phi(x)$ is increasing to $\infty$ as $x\rightarrow \infty$, we can choose $\Lambda_1$ large enough to ensure $\frac{\alpha}{q}h'(\Lambda_1)\leq 1$, i.e. $\Phi^{\frac{1}{n}}(\Lambda_1)\geq a^{-1}\cN^{\frac{1}{n}}_\Phi(F)$.

Because of the monotonicity of $\Phi$, we get in the second case \[
    F\leq -\frac{\alpha}{q}\psi_1+\Lambda_1.
\]
So we choose $f=\exp\{-\frac{\alpha}{q}\psi_1+\Lambda_1\}$, and in both cases, we have \[
    \omega_\varphi^n\leq a^n\omega_{\psi}^n+f\omega^n,
\] with \[
    ||f||_{L^q(\omega^n)}\leq C(M,\omega,a)e^{\Lambda_1}\leq C(M,\omega,q,a,\cN_{\Phi}(F),\Phi)
\]

By proposition \ref{babyrelativeestimate}, we get \[
    -\varphi-ah(-\frac{\alpha}{q}\psi_1+\Lambda_1)\leq C(M,\omega,q,a,\cN_{\Phi}(F),\Phi).
\]
 
Next we examine the functions $h$ with different choice of $\Phi$ and close the proof of the theorem.
\begin{enumerate}
\item If $\Phi(x)=|x|^p,0<p<n$, then $h(s)=a^{-1}\alpha^{-1}q\cN_\Phi(F)^{\frac{1}{n}}\frac{n}{n-p}s^{1-\frac{p}{n}}$. 

\item If $\Phi(x)=|x|^n$, $h(s)=a^{-1}\alpha^{-1}q\cN_\Phi(F)^{\frac{1}{n}}\log s$.
\item If $\Phi(x)=|x|^n(\log|x|)^p, 0<p<n,$ then $h(s)=a^{-1}\alpha^{-1}q\cN_\Phi(F)^{\frac{1}{n}}\frac{n}{n-p}(\log s)^{1-\frac{p}{n}}$, and $\Lambda_1$ is determined by $\Lambda_1(\log\Lambda_1)^{\frac{p}{n}}=a^{-1}\cN_\Phi(F)^{\frac{1}{n}}$.

In those three cases, there is no uniform bound on the function $h(-\frac{\alpha}{q}\psi_1+\Lambda_1)$ which are different from the next case.
\item If $\int_1^\infty\Phi^{-\frac{1}{n}}(x)sx<\infty$, then $h$ is bounded.
\item If $\Phi(x)$ is just increasing to $\infty$ as $x\rightarrow \infty$, then for any $\epsilon>0$, there exists constant $C_\epsilon$ such that $h(s)\leq \epsilon s+C_\epsilon$. 
\end{enumerate}
Then the integral bound of $\int_{M}\exp{(-\alpha\psi_1)}\omega^n\leq C$ gives the desired estimates in the theorem.
\end{proof}
We give an application of the above theorem to show $L^\infty$ estimate for cscK equation with an integral bound of the scalar curvature. 
\begin{corollary} Let $(M^n,\omega)$ be a compact K\"ahler manifold with and $\omega+\ddc\varphi$ be another K\"ahler metric with scalar curvature $R_\varphi$. Assume $\omega_\varphi^n=e^F\omega^n$ for some smooth function $F$ with $\int_Me^F\omega^n=\int_M\omega^n$. Let $\Phi_1(x),\Phi_2(x)$ be two positive increasing function mapping $\mathbb{R}$ to $\mathbb{R}^+$ and approach $\infty$ as $x\rightarrow \infty$. There exists a constant $C$ depending on $M,\omega, n, \int_M\Phi_1(F)e^F\omega^n, \int_M\Phi_2(\log(|R_\varphi|^ne^F))|R_\varphi|^ne^F\omega^n$ such that \[
    -\varphi-\sup\varphi\leq C.
\]
\end{corollary}

\begin{proof}The cscK equation can be written as \[\begin{aligned}
&(\omega+\ddc\varphi)^n=e^F\omega^n,\\
&\Delta_\varphi F=-R_\varphi+tr_{\varphi}Ric.
\end{aligned}\]
In the following we normalize the metric $\omega$ with $\int_M\omega^n=1$, and assume $\sup \varphi=0$.
Consider the auxiliary equations \[
 (\omega+\ddc\psi_1)^n=\frac{\Phi_1(F)e^F\omega^n}{\cN_{\Phi_1}(F)},\quad \sup\psi_1=0,
\]
and \[
    (\omega+\ddc\psi_2)^n=\frac{|R_\varphi|^ne^F\omega^n}{B}, \quad \sup\psi_2=0.
\]
Let $\psi:=-h(-\frac{\alpha}{q}\psi_1+\Lambda_1)$, and $f:=\exp\{-\frac{\alpha}{q}\psi_1+\Lambda_1\}$, with $h$ and $\Lambda_1$ are defined as in (\ref{definitionofh}) such that $\omega_\varphi^n\leq a^n\omega_\psi^n+f\omega^n.$ Define $ 
    \Psi:=F-\lambda\varphi+\frac{B^{\frac{1}{n}}}{n}\psi_2+b\psi-\epsilon(-\psi_{3,j}+\Lambda_3)^{\frac{n}{n+1}}$,
where $\psi_{3,j}$ is the solution to \[
    (\omega+\ddc\psi_{3,j})^n=\frac{\eta_j(F-\lambda\varphi+\frac{B^{\frac{1}{n}}}{n}\psi_2+b\psi)f\omega^n}{A_j},\quad \sup\psi_{3,j}=0,
\]
and $A_j:=\int_M\eta_j(F-\lambda\varphi+\frac{B^{\frac{1}{n}}}{n}\psi_2+b\psi)f\omega^n$.

Then at the point $x_0$ where the maximum of $\Psi$ is attained(without loss of generality we can assume $F-\lambda\varphi+\frac{B^{\frac{1}{n}}}{n}\psi_2+b\psi>0$ at $x_0$), we have \[
    \begin{aligned}
    0\geq &\Delta_\varphi v\\
    \geq & -R_\varphi+(\lambda-|Ric_\omega|-b-\epsilon\frac{n}{n+1}(-\psi_{3,j}+\Lambda_3)^{-\frac{1}{n+1}})tr_\varphi g-\lambda n+B^{\frac{1}{n}}\left(\frac{\omega^n_{\psi_2}}{\omega_\varphi^n}\right)^{\frac{1}{n}}+nb\left(\frac{\omega^n_{\psi}}{\omega_\varphi^n}\right)^{\frac{1}{n}}\\
    &+\epsilon\frac{n}{n+1}(-\psi_{3,j}+\Lambda_3)^{-\frac{1}{n+1}}\left(\frac{\omega^n_{\psi_{3,j}}}{\omega_\varphi^n}\right)^{\frac{1}{n}}\\
    \geq &-\lambda n+bn\left(\frac{\omega^n_{\psi}}{\omega_\varphi^n}\right)^{\frac{1}{n}}+\epsilon\frac{n}{n+1}(-\psi_{3,j}+\Lambda_3)^{-\frac{1}{n+1}}\left(\frac{\omega^n_{\psi_{3,j}}}{\omega_\varphi^n}\right)^{\frac{1}{n}}.
    \end{aligned}
\]
So we get \[
    \omega_\varphi^n\geq \lambda^{-n}b^n\omega_{\psi}^n+\lambda^{-n}\left(\frac{\epsilon}{n+1}\right)^{n}(-\psi_{3,j}+\Lambda_3)^{-\frac{n}{n+1}}\frac{\eta_j(F-\lambda\varphi+\frac{B^{\frac{1}{n}}}{n}\psi_2+\psi)f\omega^n}{A_j}
\]
Now we specify the constant $a=\lambda^{-1}b$ in defining $h$, and the above inequality implies \[
    \eta_j(F-\lambda\varphi+\frac{B^{\frac{1}{n}}}{n}\psi_2+\psi)\leq \lambda^n\left(\frac{n+1}{\epsilon}\right)^{n}A_j(-\psi_{3,j}+\Lambda_3)^{\frac{n}{n+1}}
\]
Choose\[
    \lambda=|Ric_\omega|_{\omega}+2, \epsilon=(\lambda(n+1))^{\frac{n}{n+1}}A_j^{\frac{1}{n+1}},\Lambda_3=\lambda^{n}\frac{n^{n+1}}{n+1}A_j.\]
Therefore $\Psi\leq 0$.

By theorem \ref{mosertrudingertype}, we have $||\exp\{-\psi_2\}||_{L^p(\omega^n)}$ is uniformly bounded for any fixed $p>0$ due to the bound on $\int_M\Phi_2(\log(|R_\varphi|^ne^F))|R_\varphi|^ne^F\omega^n$. With the same reason, $||e^{-b\psi}||_{L^p(\omega^n)}$ is also uniformly bounded for any fixed $p>1$. From $\Psi\leq 0$, we get  for any $p>1$ one can bound the $L^p(\omega^n)$-norm of $e^F$ by a uniform constant depending on the parameters as in the theorem. So back to the first equation in the cscK equation, we get uniform estimate for the potential $\varphi$.
\end{proof}

\begin{remark}  In the fundamental work of Chen-Cheng \cite{CC1}, the uniform estimate was proved by assuming an $L^\infty$ bound of the scalar curvature $R_\varphi$. Later, Lu and Seyyedali in \cite{lu2023remarks} proved such estimate with the a bound of $\int_M|R_\varphi|^pe^F\omega^n$ for some $p>n$. Our corollary says that the uniform estimate holds under a bound on $\int_M\Phi_2(|R_\varphi|^ne^F)|R_\varphi|^ne^F\omega^n$.  We remark that under the additional assumption on the uniform Sobolev constant and $\int_M|R_\varphi|^pe^F\omega^n$ for some $p>n$, Li and Zhang \cite{li2023calabiflowboundedlp} proved higher order estimates for the cscK equations.
\end{remark}

\section{Modulus of continuity of solutions}
The above method can also be used to prove the following theorem on the modulus of continuity of solutions to complex Monge-Amp\`ere equations, which has the same strength as proved by pluripotential theory (see \cite{GGZ}):
\begin{theorem}Let $(M,\omega)$ be a compact K\"ahler manifold with a fixed K\"ahler metric $\omega$. Fix $p>n$, let $u$ be a smooth solution to the equation \[(\omega+\ddc u)^n=e^F\omega^n, \quad \omega+\ddc u>0.
\]There exist constants $C_1=C_1(n,p,\omega,||e^F||_{L^1(\log L)^p}),C_2=C_2(n,p,\omega,||e^F||_{L^1(\log L)^n(\log\log L)^p})$ such that \begin{equation}
|u(x)-u(y)|\leq \frac{C_1}{|\log d_{\omega}(x,y)|^{\frac{p}{n}-1}},
\end{equation}
and \begin{equation}
|u(x)-u(y)|\leq \frac{C_2}{\left(\log(-\log d_{\omega}(x,y))\right)^{\frac{p}{n}-1}}.
\end{equation}
\end{theorem}

This theorem generalizes Guo-Phong-Tong-Wang's results that we can choose the power $\alpha$ in the modulus of continuity $m(r)=\frac{C}{|\log r|^\alpha}$ to be $\alpha=\frac{p}{n}-1$ rather than $\alpha=\min(\frac{p}{n}-1,\frac{p}{n+1})$. And we'll see how the auxiliary function helps us and the choice of constant $h,\Lambda_1$ determines the choice of modulus of continuity. 

To show this, we need some facts about Demailly's approximation of psh functions.
Let $\rho:\mathbb{C}^n\rightarrow \mathbb{R}$ be a smooth kernel such that $\int_{\mathbb{C}^n}\rho(|w|)dV(w)=1$ and $\rho$ has compact support in the unit ball. One define \[
    \rho_\delta u(x):=\frac{1}{\delta^{2n}}\int_{T_xM}u(\exp_x(\zeta))\rho(\frac{|\zeta|_{\omega}^2}{\delta^2})dV_{\omega}(\zeta)
\]
We list some properties of $\rho_\delta u$ we need in the proof as follows:
\begin{proposition}[Berman-Demailly, lemma 1.2 in \cite{MR2884031}, see also \cite{GGZ}]\label{l1approximation}Let $u\in psh(M,\omega)$ with normalization $1\leq u\leq C$. We have \begin{enumerate}
\item There exists a constant $K>0$ depends on the curvature of $\omega$, such that $\rho_tu+Kt$ is monotone increasing in $t$. 
\item For any $c>0$, define $U_{c,\delta}=\inf_{t\in(0,\delta)}(\rho_tu(x)+Kt-c\log(\frac{t}{\delta})-K\delta)$, then \[
    \omega+\ddc U_{c,\delta}\geq -(Ac+K\delta)\omega,
\]where $-A$ is the lower bound of the bidectional curvature of $\omega$.
\item $\int_M|\rho_\delta u-u|\omega^n\leq C(n,\omega)\delta^2, \quad \forall \delta\in(0,1].$
\end{enumerate}
\end{proposition}
Now we start to prove the theorem.
\begin{proof}
We choose a small $c>0$ but large compared with $\delta$. Then $\omega+\ddc U_{c,\delta}\geq -A'c\omega$. So we can define a $\omega$-psh function $u_{c,\delta}:=(1-A'c)U_{c,\delta}$. So we have $\omega+\ddc u_{c,\delta}\geq (A'c)^2\omega\geq 0.$ And from the normalization of $u$ we have $u_{c,\delta}>0$. 

To use the relative $L^\infty$-estimate, let's state the comparison functions in the previous section as a lemma with an address on the constants. For notation-simplicity, we normalize the volume $\int_M\omega^n=1$.
\begin{lemma}\label{choiceofe1lambda1}Let $\psi_1$ be the solution to the equation\[
    (\omega+\ddc\psi_1)^n=\frac{\Phi(F)e^F\omega^n}{\cN_\Phi(F)},\quad \sup\psi_1=0.
\]\begin{enumerate}
\item If $\Phi(x)=|x|^p$ with $p>n$, we choose \[
h(s)=-\frac{n}{p-n}a^{-1}\alpha^{-1}q\cN_\Phi(F)^{\frac{1}{n}}s^{1-\frac{p}{n}},\quad \Lambda_1=\left(a^{-1}\cN_\Phi(F)^{\frac{1}{n}}\right)^{\frac{n}{p}}.\]
And $\psi=-h(-\frac{\alpha}{q}\psi_1+\Lambda_1)$.
Then \[
    e^F\omega^n\leq a^n\omega_\psi^n+f\omega^n,
\]where $f=\exp\{-\frac{\alpha}{q}\psi_1+\Lambda_1\}$. 
\item If $\Phi(x)=|x|^n(\log|x|)^p, p>n.$ We choose 
\[h(s)=-\frac{n}{p-n}a^{-1}\alpha^{-1}q\cN_\Phi(F)^{\frac{1}{n}}(\log s)^{1-\frac{p}{n}}, \quad \Lambda_1(\log\Lambda_1)^{\frac{p}{n}}=a^{-1}\cN_\Phi(F)^{\frac{1}{n}}.
\]
Let $\psi=-h(-\frac{\alpha}{q}\psi_1+\Lambda_1)$. Then \[
    e^F\omega^n\leq a^n\omega_\psi^n+f\omega^n,
\]where $f=\exp\{-\frac{\alpha}{q}\psi_1+\Lambda_1\}$. 
\end{enumerate}
 \end{lemma} 
 In the following, we'll use the notations:\[
    \begin{aligned}
    &\Omega_s=\{-u+(1-2r)u_{c,\delta}+r\psi-s>0\}\\
    &A_{s,j}=\int_{\Omega_s}\eta_j(-u+(1-2r)u_{c,\delta}+r\psi-s)f\omega^n\\
    &\phi(s)=\int_{\Omega_s}f\omega^n
    \end{aligned}
\]
 Now we consider the function \[
 \Psi_j(x):=(-u+(1-2r)u_{c,\delta}+r\psi-s)-\epsilon_{2,j}(-\psi_{2,j}+\Lambda_{2,j})^{\frac{n}{n+1}},
 \]
where $\psi_{2,j}$ is the solution to the equation \[
 (\omega+\ddc\psi_{2,j})^n=\frac{\eta_j(-u+(1-2r)u_{c,\delta}+r\psi-s)f\omega^n}{A_{s,j}};\quad \sup\psi_{2,j}=0,
\]
and $\psi$ is constructed by lemma \ref{choiceofe1lambda1} with a choice of $a=r$.
We also choose $\epsilon_{2,j},\Lambda_{2,j}$ as 
\[
    \epsilon_{2,j}=\left(\frac{n+1}{n}\right)^{\frac{n}{n+1}}A_{s,j}^{\frac{1}{n+1}},\quad \Lambda_{2,j}=\frac{n}{n+1}\frac{1}{r^{n+1}}A_{s,j}.
\]

We claim that $\Psi_{j}\leq 0$. Assume $\Psi_{j}$ attains its maximum at $x_0$. We can assume $x_0\in \Omega_{s}$, otherwise we already have $\Psi_{j}\leq 0$. Then at $x_0$, we have (for convenience we suppress the subscript $j$)\[
    \begin{aligned}
    0\geq &-\ddc u+(1-2r)\ddc u_{c,\delta}+r\ddc\psi+\epsilon_{2}\frac{n}{n+1}(-\psi_2+\Lambda_2)^{-\frac{1}{n+1}}\ddc\psi_2\\
    &+\epsilon_2\frac{n}{(n+1)^2}(-\psi_2+\Lambda_2)\sqrt{-1}\partial\psi_2\wedge\bar{\partial}\psi_2\\
    \geq&-\omega_{u}+(1-2r)\omega_{u_{c,\delta}}+r\omega_{\psi}+\epsilon_2\frac{n}{n+1}(-\psi_2+\Lambda_2)^{-\frac{1}{n+1}}\omega_{\psi_2}
\end{aligned}
\]
In the last equality we used $\epsilon_{2,j}\frac{n}{n+1}(\Lambda_{2,j})^{-\frac{1}{n+1}}=r$, hence $r\omega\geq \epsilon_{2,j}\frac{n}{n+1}(-\psi_{2,j}+\Lambda_{2,j})^{-\frac{1}{n+1}}\omega$.

Now we apply the comparison lemma \ref{choiceofe1lambda1}, \[
\begin{aligned}
  r^n\omega_\psi^n+f\omega^n\geq \omega_u^n\geq &\left((1-2r)\omega_{u_{c,\delta}}+r\omega_{\psi}+\epsilon_{2,j}\frac{n}{n+1}(-\psi_{2,j}+\Lambda_{2,j})^{-\frac{1}{n+1}}\omega_{\psi_{2,j}}\right)^n\\
  \geq &r^n\omega_\psi^n+\left(\epsilon_{2,j}\frac{n}{n+1}\right)^n(-\psi_{2,j}+\Lambda_{2,j})^{-\frac{1}{n+1}}\frac{\eta_j(-u+(1-2r)u_{c,\delta}+r\psi-s)}{A_{s,j}}f\omega^n
\end{aligned}
\]
It's equivalent to $\Psi_j\leq 0$ with our choice of $\epsilon_{2,j}, \Lambda_{2,j}$.

Next, we want to choose the constant $\Lambda_1$, $c,r$ to get a good estimate. 

\begin{lemma} \label{controlofphi(s0)byapproximationofu}\begin{enumerate}
\item If $\Phi(x)=|x|^p,p>n$, for any $q>1, q^*=\frac{q}{q-1}$, choose $s_0$ as \[
    s_0=2\frac{n}{p-n}\alpha^{-1}q\cN_\Phi(F)^{\frac{1}{n}}(\Lambda_1)^{1-\frac{p}{n}}\]
Then \[
    \phi(s_0)\leq C(M,\omega,q,p,n,\cN_\Phi(F))e^{\Lambda_1}s_0^{-\frac{1}{q^*}}||(u_{c,\delta}-u)_+||_{L^1(\omega^n)}^{\frac{1}{q^*}}
\]
\item If $\Phi(x)=|x|^n(\log x)^{p}, p>n$, for any $q>1, q^*=\frac{q}{q-1}$Choose $s_0$ as \[
 s_0=2\frac{n}{p-n}{\alpha^{-1}q\cN_\Phi(F)^{1/n}}(\log\Lambda_1)^{1-p/n}\]
Then \[
    \phi(s_0)\leq C(M,\omega,q,p,n,\cN_\Phi(F))e^{\Lambda_1}s_0^{-\frac{1}{q^*}}||(u_{c,\delta}-u)_+||_{L^1(\omega^n)}^{\frac{1}{q^*}}\]
\end{enumerate}
\end{lemma}
\begin{proof}
The proofs for different $\Phi(x)$ are similar, we prove the the case $\Phi(x)=|x|^n(\log|x|)^p$ for example. 

From the definition of $\Omega_{s_0}$ and the fact $u_{c,\delta}>0$ because of the normalization of $u$, we have $$-u+u_{c,\delta}+r\psi>s_0.$$ By lemma \ref{choiceofe1lambda1} with $a=r$, we have $$0<r\psi\leq =\frac{p}{p-n}\alpha^{-1}q\cN_\Phi(F)^{\frac{1}{n}}\left(\log(-\frac{\alpha}{q}\psi_1+\Lambda_1)\right)^{1-p/n}\leq \frac{n}{p-n}\alpha^{-1}q\cN_\Phi(F)^{\frac{1}{n}}(\log\Lambda_1)^{1-\frac{p}{n}}.$$
 With the choice of $s_0$, we have $0<r\psi\leq \frac{1}{2}s_0$. Then on $\Omega_{s_0}$, we have \[
    u_{c,\delta}-u\geq s_0-r\psi\geq \frac{s_0}{2}.
\]
As a consequence, we have \[
    \phi(s_0)=\int_{\Omega_{s_0}}f\omega^n\leq \int_M\frac{(u_{c,\delta}-u)_+^{1/q^*}}{(\frac{s_0}{2})^{1/q^*}}f\omega^n\leq \frac{C}{s_0^{1/q^*}}||(u_{c,\delta}-u)_+||_{L^1(\omega^n)}^{1/q^*}||f||_{L^q(\omega^n)},
\]
where $\frac{1}{q}+\frac{1}{q^*}=1$.

Since $||f||_{L^q(\omega^n)}\leq C(M,\omega)e^{\Lambda_1}$, so \[
    \phi(s_0)\leq C(M,\omega,q,p,n,\cN_\Phi(F))e^{\Lambda_1}s_0^{-\frac{1}{q^*}}||(u_{c,\delta}-u)_+||^{\frac{1}{q^*}}_{L^1(\omega^n)}
\]
\end{proof}

Now we go to the iteration process. As before we have\[
    t\phi(t+s)\leq A_{s},
\]
and \[
    A_{s}\leq c_nA_{s,j}^{\frac{1}{n+1}}\int_M(-\psi_{2,j}+\Lambda_{2,j})^{\frac{n}{n+1}}f\omega^n.
\]

For $s\geq s_0$ we have \begin{equation}\label{estimatelambda2}
\Lambda_{2,j}=\frac{n}{n+1}\frac{A_{s,j}}{r^{n+1}}\leq C(M,\omega,n,p,q,\cN_\Phi(F))\frac{\phi(s_0)}{r^{n+1}}\leq C(M,\omega,n,p,q,\cN_\Phi(F))e^{\Lambda_1}s_0^{-\frac{1}{q^*}}\frac{||(u_{c,\delta}-u)_+||_{L^1(\omega^n)}^{\frac{1}{q^*}}}{r^{n+1}} \end{equation}

So if we choose $r$ which ensures  $\Lambda_{2,j}$ bounded uniformly, then 
\[
\begin{aligned}
A_{s,k}&\leq c_nA_{s,j}^{\frac{1}{n+1}}\left(\int_{\Omega_s}(-\psi_{2,j}+\Lambda_{2,j})^{\frac{nq_2}{n+1}}f\omega^n\right)^{\frac{1}{q_2}}\phi(s)^{1-\frac{1}{q_2}}\\
&\leq c_nA_{s,j}^{\frac{1}{n+1}}\left(\int_{M}(-\psi_{2,j}+\Lambda_{2,j})^{\frac{nq_2q^*}{n+1}}\omega^n\right)^{\frac{1}{q_2q^*}}||f||_{L^q(\omega^n)}^{\frac{1}{q_2}}\phi(s)^{1-\frac{1}{q_2}}
\end{aligned}\] 
Let $j\rightarrow \infty$, we get \[
    A_s\leq C(M,\omega,n,p,q,q_2,\cN_\Phi(F))e^{\frac{1}{q_2}\Lambda_1}A_s^{\frac{1}{n+1}}\phi(s)^{1-\frac{1}{q_2}}.
\]
Let $\delta_0=(1-\frac{1}{q_2})\frac{n+1}{n}-1>0$ by fix $q_2>n+1$.
Then we get \[
    A_s\leq C(M,\omega,n,p,q,q_2,\cN_\Phi(F))e^{\frac{1}{q_2}\Lambda_1}\phi(s)^{1+\delta_0}, \text{ for }s\geq s_0.
\]
By De Giorgi's iteration lemma \ref{DeGiorgiiterationlemma} we get $
    \phi(s_\infty)=0,$ and \begin{equation}\label{estimatesinfty}
    \begin{aligned}
     s_\infty&\leq \frac{2}{1-2^{-\delta_0}}C(M,\omega,n,p,q,q_2,\cN_\Phi(F))e^{\frac{1}{q_2}\Lambda_1}\phi(s_0)^{\delta_0}+s_0.\\
     &\leq C(M,\omega,n,p,q,\cN_\Phi(F))e^{(\delta_0+\frac{1}{q_2})\Lambda_1}s_0^{-\frac{\delta_0}{q^*}}||(u_{c,\delta}-u)_+||_{L^1(\omega^n)}^{\frac{\delta_0}{q^*}}+s_0.
\end{aligned}
\end{equation}In the second inequality we have used lemma \ref{controlofphi(s0)byapproximationofu}.

Now, we can choose $r$ as follows: 
First choose $\Lambda_1=-\frac{1}{B}\log\delta$, for some large $B$ which depends on $n,q$ only. Then for $r$,
\begin{enumerate}
\item If $\Phi(x)=|x|^p,p>n,$ then $r$ is determined by $\Lambda_1=(r^{-1}\cN_\Phi(F)^{\frac{1}{n}})^{\frac{n}{p}}$. 
\item If$\Phi(x)=|x|^n(\log|x|)^{p},p>n,$, then $r$ is determined by $  \Lambda_1(\log\Lambda_1)^{\frac{p}{n}}=r^{-1}\cN_\Phi(F)^{\frac{1}{n}}$.
\end{enumerate}

With such a choice of $r,\Lambda_1$, there exists a uniform constant $C$ such that $s_0\leq C|\log\delta|^{1-\frac{p}{n}}$ when $\Phi(x)=|x|^p$, and $s_0\leq C|\log(-\log\delta)|^{1-\frac{p}{n}}$ when $\Phi(x)=|x|^n(\log|x|)^{p}$. Therefore from equation (\ref{estimatelambda2}) and proposition \ref{l1approximation} that $||(u_{c,\delta}-u)_+||_{L^1(\omega^n)}\leq C\delta^2$, there exists a uniform constant such $C$ such that  $\Lambda_{2,j}\leq C$. 
Moreover, from equation (\ref{estimatesinfty}), \[
    s_\infty\leq C(M,\omega,q,p,n,\cN_\Phi(F))|\log\delta|^{1-\frac{p}{n}}, \quad \text{for }\Phi(x)=|x|^p,
\]and \[
    s_\infty\leq C(M,\omega,q,p,n,\cN_\Phi(F))|\log(-\log\delta)|^{1-\frac{p}{n}}, \quad \text{for } \Phi(x)=|x|^n(\log|x|)^{p}.
\]

Next, we choose $c=|\log\delta|^{1-\frac{p}{n}}, \text{for }\Phi(x)=|x|^p$, and $c=|\log(-\log\delta)|^{1-\frac{p}{n}},\text{for } \Phi(x)=|x|^n(\log|x|)^{p}.$ We claim that there exists a uniform constant $\delta_1$, such that $$\rho_\delta u(x)-u(x)\leq C(M,\omega,n,p,q,\cN_\Phi(F))c, \quad \forall \delta\leq \delta_1.$$

Since $\psi>0$, we have \[
    u_{c,\delta}-u\leq 2ru_{c,\delta}+Cc.
\]
From the definition of $u_{c,\delta}=(1-Ac)U_{c,\delta}$ and $$U_{c,\delta}=\inf_{t\in(0,\delta]}\{\rho_t u+Kt-c\log(\frac{t}{\delta})-K\delta\},$$
we have there exists $t_x\in (0,\delta]$ such that \[
    \rho_{t_x}u(x)+Kt_x-c\log(\frac{t_x}{\delta})-K\delta=U_{c,\delta}(x)\leq u(x)+A'cU_{c,\delta}(x)+2ru_{c,\delta}+Cc.
\]
Since $\rho_tu+Kt$ is increasing in $t$, we have $\rho_{t_x}u(x)+Kt_x\geq u(x)$. 
So \[
    -\log(\frac{t_x}{\delta})\geq K\frac{\delta}{c}+A'U_{c,\delta}+\frac{2r}{c}u_{c,\delta}+C.
\]
From the choice of $r,s_0,c$ we have \[
    -\log\frac{t_x}{\delta}\geq -C.
\] In other words, there exists a uniform constant $\theta$ such that $t_x\geq \theta\delta$. So using the fact that $\rho_tu+KT$ is increasing in $t$ again, and the observation $\delta\leq c$, we have \[
\begin{aligned}
  \rho_{\theta\delta}+K\theta t-u\leq &\rho_{t_x}+Kt_x-u\\
  \leq & c\log(\frac{t_x}{\delta})+K\delta+A'cU_{c,\delta}+2ru_{c,\delta}+Cc\\
  \leq &Cc.
\end{aligned}
\]
It implies that there is a universal constant $\delta_1>0$ such that $\forall \delta\leq  \delta_1$, we have \[
    \rho_\delta u(x)-u(x)\leq Cc=\left\{\begin{aligned}
    &\frac{C}{|\log\delta|^{\frac{p}{n}-1}} & \text{ when } 
       \Phi=|x|^p, p>n,\\
        &\frac{C}{(\log(-\log\delta))^{p/n-1}}, & \text{ when } \Phi(x)=|x|^n\log^{p}(|x|), p>n. 
    \end{aligned}\right.
\]

The left is standard, c.f. \cite{guo2021modulus,zeriahi2020}.

\end{proof}

\section{Stability estimate}
We can use the relative $L^\infty$-estimate to prove a stability result also. The key observation is that we can refine our choice of $f$ in the comparison inequality $\omega_v^n\leq a^n\omega_{\psi}^n+f\omega^n$. In previous section, we just choose $$f=\exp\{-\frac{\alpha}{q}\psi_1+\Lambda_1\}.$$
In fact, we can choose $\tilde{f}=\min(\frac{\omega_v^n}{\omega^n},f)$. Then we have additional condition $f\omega^n\leq \omega_v^n$. This is important in the proof of the decay rate of the quantity $\int_{E_s}f\omega^n$. With this one can generalize the stability estimate under a weaker assumption of Orlicz norm of the volume form. 
In the following, we say that $\varphi$ is an admissible solution to \begin{equation}\label{hessianequations}
    (\omega_t+\ddc\varphi)^k\wedge\omega_M^{n-k}=c_te^F\omega_M^n\end{equation} if $\lambda[\omega_t+\ddc\varphi]\in \Gamma_k$, where $\Gamma_k:=\{\lambda\in \RR^n|\sigma_i(\lambda)\geq 0, i=1,\cdots, k\}$, and $\lambda[\alpha]$ is the eigenvalues of a real $(1,1)$-form $\alpha$ with respect to $\omega_M$.
    
\begin{theorem}Let $(M,\omega_M)$ be a compact K\"ahler manifold with a given K\"ahler metric $\omega_M$ with normalization $[\omega_M]^n=1$. Let $\Phi$ be an increasing positive function such that $\int^\infty\Phi^{-\frac{1}{n}}(t)dt<\infty,$ and $\Phi(x)\leq e^{C_0x}$ for all $x\geq 0$.
Given $K>0$, let $G,H$ be smooth functions in\[
    \mathcal{K}:=\left\{F|\int_Me^F\omega_M^n=\int_M\omega_M^n, \int_M\Phi(|F|)e^{\frac{n}{k}F}\omega_M^n\leq K\right\}.
\]
Let $\chi$ be a closed nonnegative real $(1,1)$-form on $M$, and define $\omega_t=\chi+t\omega_M$. Let $u,v$ be smooth admissible solutions to \[
    (\omega_t+\ddc u)^k\wedge\omega_M^{n-k}=c_{t}e^G\omega_M^n,\quad (\omega_t+\ddc v)^k\wedge\omega_M^{n-k}=c_te^H\omega_M^n,
\]where $c_t:=\int_M\omega_t^k\wedge\omega_M^{n-k}$.  We normalize $u,v$ as \[
    \max(u-v)=\max(v-u).
\]There exists a constant $C(||e^G-e^H||_{L^1(\omega_M^n)}|M,\omega_M,n,K,\frac{c_t^{\frac{1}{k}}}{V_t^{\frac{1}{n}}},\Phi,C_0)$ such that \[
    \max(u-v)\leq C(||e^G-e^H||_{L^1(\omega_M^n)}|M,\omega_M,n,K,\frac{c_t^{\frac{1}{k}}}{V_t^{\frac{1}{n}}},\Phi,C_0),
\]
where $V_t:=\int_M\omega_t^n.$
\end{theorem}
In the above theorem we use the notation that $C(t|a,b,c,\ldots)$  to mean a constant depending on $t,a,b,c,\ldots$ and \[
    \lim_{t\rightarrow 0}C(t|a,b,c,\ldots)=0.
\]
In fact the constant $C(t|a,b,\cdots,)$ in the theorem can be made explicit as  \begin{equation}\label{expliciteexpresionconstant}
 C\int_{-\frac{1}{C}\log||e^G-e^H||_{L^1(\omega_M^n)}}^\infty \Phi^{-\frac{1}{n}}(x)dx+C\Phi^{-\frac{1}{n}}(-\frac{1}{C}\log(||e^G-e^H||_{L^1(\omega_M^n)})) +C||e^G-e^H||^\eta_{L^1(\omega_M^n)},
\end{equation} for a uniform constant $\eta>0,C>0$ depending on $M,\omega_M,n,K,\frac{c_t^{\frac{1}{k}}}{V_t^{\frac{1}{n}}},\Phi, C_0$. 

When $\Phi(x)=x^n\log^p(x) $ for $x>2$, and $p>n,$ the constant $C(t|a,b,\cdots)>0$ can be chosen to be \[
\frac{C}{\log^{\frac{p}{n}-1}(-\log||e^G-e^{H}||_{L^1(\omega_M^n)})}, \text{ for some uniform constant}  C>0.
\]
\begin{proof}
We suppress the subscript $t$ in $\omega_t$ in the following. 
Consider the auxiliary equation\[
    (\omega+\ddc\psi_1)^n=V_t\frac{e^{\frac{n}{k}H}\Phi(H)\omega_M^n}{\cN_\Phi},\quad \sup\psi_1=0,
\]
where $\cN_\Phi:=\int_Me^{\frac{n}{k}H}\Phi(H)\omega_M^n$, and $V_t=\int_M\omega^n$.
For any $a\in(0,1),q>1$, and $\alpha$ denoting an upper bound for the $\alpha$-invariant for metric $\omega_t$, define function $h(s):\mathbb{R}\rightarrow \mathbb{R}$ as \[
    h(s)=-\frac{c_t^{\frac{1}{k}}}{V_t^{\frac{1}{n}}}\int_s^\infty a^{-1}\alpha^{-1}q\cN_\Phi^{\frac{1}{n}}\Phi(t)^{-\frac{1}{n}}dt.
\]
Choose $\Lambda_1$ such that $\frac{\alpha}{q}h'(\Lambda_1)\leq1$, then one has $\psi:=-h(-\frac{\alpha}{q}\psi_1+\Lambda_1)\in Psh(M,\omega)$, \[
    \omega+\ddc\psi\geq \frac{\alpha}{q}h'(-\frac{\alpha}{q}\psi_1+\Lambda_1)(\omega+\ddc\psi_1).
\]
Therefore, \[
    \left(\frac{\omega_\psi^n}{\omega_M^n}\right)^{\frac{k}{n}}\geq \frac{\alpha^k}{q^k}h'^k\left(\frac{\omega_{\psi_1}^n}{\omega_M^n}\right)^{\frac{k}{n}}=\frac{\alpha^kh'^k V_t^{\frac{k}{n}}\Phi^{\frac{k}{n}}(H)}{q^k\cN^{\frac{k}{n}}}e^H.
\] 
And a similar argument in section 2 shows\begin{equation}\label{sigmakcomparison}
    \omega_v^k\wedge\omega_M^{n-k}=c_te^H\omega_M^n\leq a^k\left(\frac{\omega_\psi^n}{\omega_M^n}\right)^{\frac{k}{n}}\omega_M^n+c_tf\omega_M^n,
\end{equation}with $f=\min\{\exp\{-\frac{\alpha}{q}\psi_1+\Lambda_1\},e^H\}$.

Next, we consider another auxiliary equation\[
    (\omega+\ddc\psi_{2,j})^{n}=V_t\frac{\eta_j(-v+(1-r)u+a\psi-s)f^{\frac{n}{k}}\omega_M^n}{A_{s,j}},\quad \sup\psi_{2,j}=0.
\]
Here $A_{s,j}:=\int_M\eta_j(-v+(1-r)u+a\psi-s)f^{\frac{n}{k}}\omega_M^n,$ and $\eta_j(x)$ is a sequence of smooth positive functions approximating $x_+$ from above. It's clear that $\lim_{j\rightarrow \infty}A_{s,j}=A_s:=\int_{E_s}(-v+(1-r)u+a\psi-s)_+f^{\frac{n}{k}}\omega_M^n.$

Now, consider the function \[
    \Psi_{j}:=(-v+(1-r)u+a\psi-s)-\epsilon_{2,j}(-\psi_{2,j}+\Lambda_{2,j})^{\frac{n}{n+1}}.
\]
The constant $\epsilon_{2,j}$ and $\Lambda_{2,j}$ are chosen as
\begin{equation}\label{sigmakchoiceofepsilonandlambda}
    \epsilon_{2,k}=\left(\frac{c_t^{\frac{1}{k}}}{V_t^{\frac{1}{n}}}\right)^{\frac{n}{n+1}}\left(\frac{n+1}{n}\right)^{\frac{n}{n+1}}A_{s,j}^{\frac{1}{n+1}},\quad \Lambda_{2,j}=\left(\frac{c_t^{\frac{1}{k}}}{V_t^{\frac{1}{n}}}\right)^n\frac{n}{n+1}\frac{1}{(r-a)^{n+1}}A_{s,j}.
\end{equation}
We claim that $\Psi_{s,j}\leq 0$. Let $x_0$ be the point where $\Psi_{s,j}$ attains its maximum, we can assume $x_0\in E_s:=\{-v+(1-r)u+a\psi-s\geq 0\}$, otherwise we already have $\Psi_{s,j}\leq 0$.
At $x_0$, we have \[
    \begin{aligned}
    0\geq &-\ddc v+(1-r)\ddc u+a\ddc\psi+\epsilon_{2,j}\frac{n}{n+1}(-\psi_{3,j}+\Lambda_{2,j})^{-\frac{1}{n+1}}\\
    \geq&-\omega_v+(1-r)\omega_u+a\omega_{\psi}+\epsilon_{2,j}\frac{n}{n+1}(-\psi_{3,j}+\Lambda_{2,j})^{-\frac{1}{n+1}}\omega_{\psi_{2,j}}
    \end{aligned}
\]
Therefore\[\begin{aligned}
\omega_v^k\wedge\omega_M^{n-k}\geq &\left((1-r)\omega_u+a\omega_{\psi}+\frac{\epsilon_{2,j}n}{n+1}(-\psi_{2,j}+\Lambda_{2,j})^{-\frac{1}{n+1}}\omega_{\psi_{2,j}}\right)^{k}\wedge\omega_M^{n-k}\\
\geq & a^k(\omega_\psi)^k\wedge\omega_M^{n-k}+\left(\frac{\epsilon_{2,j}n}{n+1}\right)^k(-\psi_{2,j}+\Lambda_{2,j})^{-\frac{k}{n+1}}(\omega_{\psi_{2,j}})^k\wedge\omega_M^{n-k}\\
\geq &a^k\left(\frac{\omega_{\psi}^n}{\omega_M^n}\right)^{\frac{k}{n}}\omega_M^n+\left(\frac{\epsilon_{2,j}n}{n+1}\right)^k(-\psi_{2,j}+\Lambda_{2,j})^{-\frac{k}{n+1}}\left(\frac{\omega_{\psi_{2,j}}^n}{\omega_M^n}\right)^{\frac{k}{n}}\omega_M^n .
\end{aligned}
\]
On the other hand, by \eqref{sigmakcomparison}, we get \[
    c_tf\omega_M^n+a^k\left(\frac{\omega_{\psi}^n}{\omega_M^n}\right)^{\frac{k}{n}}\omega_M^n\geq a^k\left(\frac{\omega_{\psi}^n}{\omega_M^n}\right)^{\frac{k}{n}}\omega_M^n+\left(\frac{\epsilon_{2,j}n}{n+1}\right)^k(-\psi_{2,j}+\Lambda_{2,j})^{-\frac{k}{n+1}}\left(\frac{\omega_{\psi_{2,j}}^n}{\omega_M^n}\right)^{\frac{k}{n}}\omega_M^n.
\]
Hence \[
    \eta_{j}(-v+(1-r)u+a\psi-s)\leq \left(\frac{c_t^{\frac{1}{k}}(n+1)}{V_t^{\frac{1}{n}}n\epsilon}\right)^{n}(-\psi_{2,j}+\Lambda_{2,j})^{\frac{n}{n+1}}.
\]
With the choice of $\epsilon_{2,j},\Lambda_{2,j}$ as \eqref{sigmakchoiceofepsilonandlambda}, the above inequality is equivalent to \[
    \eta_j(-v+(1-r)u+a\psi-s)\leq \epsilon_{2,j}(-\psi_{2,j}+\Lambda_{2,j})^{\frac{n}{n+1}}.
\]

In the following, we use $\beta$ to denote a uniform upper bound for solutions to equations \[
    (\omega_t+\ddc\varphi)^k\wedge\omega_M^{n-k}=c_te^F\omega_M^n,\quad \sup\varphi=0,
\] with $F\in \mathcal{K}$. Let $\phi(s):=\int_{E_s}f^{\frac{n}{k}}\omega_M^n.$

\begin{lemma}
There exists a uniform constant $B>0$ such that when $s>s_0=\sup_{M}(a\psi)+3\beta r$, and $\Lambda_1\leq -\frac{1}{B}\log(||e^G-e^H||_{L^1(\omega_M^n)})$, we have \[
t\phi(s+t)\leq Ce^{C\Lambda_1}\phi(s)^{1+\delta}, \forall t>0,
\]for a uniform constant $C$ and $\delta>0$.
\end{lemma}
\begin{proof}

First, we use the lemma \ref{controlAs} to uniformly bound $\Lambda_{2,j}$ as \begin{equation}\label{uniformLambda2j}
\begin{aligned}
    \Lambda_{2,j}=& \left(\frac{c_t^{\frac{1}{k}}}{V_t^{\frac{1}{n}}}\right)^n\frac{n}{n+1}\frac{A_s}{(r-a)^{n+1}}\\
    = & \left(\frac{c_t^{\frac{1}{k}}}{V_t^{\frac{1}{n}}}\right)^nc_n\frac{\int_{E_s}\eta_{j}(-v+(1-r)u+a\psi-s)f^{\frac{n}{k}}\omega^n}{(r-a)^{n+1}}\\
    \leq & \left(\frac{c_t^{\frac{1}{k}}}{V_t^{\frac{1}{n}}}\right)^nC(n,\beta,C_0)\frac{1}{(r-a)^{n+1}}\int_{E_s}f^{\frac{n}{k}}\omega_M^n\\
    \leq & \left(\frac{c_t^{\frac{1}{k}}}{V_t^{\frac{1}{n}}}\right)^n\frac{1}{(r-a)^{n+1}}\left(\int_{E_s}f^{(\frac{n}{k}-b)\frac{1}{1-b}}\omega_M^N\right)^{1-b}\left(\int_{E_s}f\omega_M^n\right)^b\\
    \leq&  \left(\frac{c_t^{\frac{1}{k}}}{V_t^{\frac{1}{n}}}\right)^nc(n,\beta,C_0)e^{q(1-b)\Lambda_1}\frac{||e^G-e^H||^b_{L^1(\omega_M^n)}}{(r-a)^{n+1}r^{3b}}, \text{ for any } s\geq s_0,
\end{aligned}
\end{equation}
where $b\in(0,1)$ is chosen such that $(\frac{n}{k}-b)\frac{1}{1-b}= q$, i.e. $b=\frac{q-\frac{n}{k}}{q-1}$.

We require that $\frac{\alpha}{q}h'(\Lambda_1)= 1$ to get the comparison metric, it implies  $\Phi(\Lambda_1)^{\frac{1}{n}}\geq Cr^{-1}$ for some uniform constant $C$. Note that we assume that $\Phi(x)\leq e^{C_0x}$ for $x\geq 0$, so we get $r^{-1}\leq e^{C\Lambda_1}$ for some uniform constant $C$. Now we choose $a=\frac{r}{4},$ plug it into \eqref{uniformLambda2j}, we find there exists a constant $B$ such that if $s>s_0=\sup_{M}(a\psi)+3\beta r$, and $\Lambda_1\leq -\frac{1}{B}\log(||e^G-e^H||_{L^1(\omega_M^n)})$, then $\Lambda_{2,j}$ is uniformly bounded. 

Next we estimate $A_{s,j}$ as \[
\begin{aligned}
A_{s}\leq  &C(n,\beta,C_0,\frac{c_t^{\frac{1}{k}}}{V_t^{\frac{1}{n}}})A_{s,j}^{\frac{1}{n+1}}\int_{\Omega_s}(-\psi_{2,j}+\Lambda_{2,j})^{\frac{n}{n+1}}f^{\frac{n}{k}}\omega_M^n\\
\leq &C_1A_{s,j}^{\frac{1}{k+1}}\left(\int_{M}(-\psi_{2,j}+\Lambda_{2,j})^{\frac{np}{n+1}}f^{\frac{n}{k}}\omega_M^n\right)^{1/p}(\int_{E_s}f^{\frac{n}{k}}\omega_M^n)^{1-1/p}.\\
\leq &C_2e^{C_3\Lambda_1}A_{s,j}^{\frac{1}{n+1}}\phi(s)^{1-\frac{1}{p}}, \text{ when } \Lambda_{2,j} \text{ is uniformly bounded.}
\end{aligned}
\]
In the above we used the estimate of \[
\int_{M}(-\psi_{2,j}+\Lambda_{2,j})^{\frac{np}{n+1}}f^{\frac{n}{k}}\omega_M^n\leq ||(-\psi_{2,j}+C)^{\frac{np}{n+1}}||_{L^{p_2}(\omega_M^n)}||f||_{L^q(\omega_M^n)}^{\frac{n}{kq}}\leq Ce^{C_3\Lambda_1}, \text{ for }p_2=\frac{qk}{qk-n}.
\]
Choose arbitrary $p>n+1$ and $\delta_0=(1-\frac{1}{p})\frac{n+1}{n}-1>0$, and let $j\rightarrow \infty$, we get \[
  A_s\leq C_2e^{C_3\Lambda_1}\phi(s)^{1+\delta}, \quad \text{for } s\geq s_0, \text{ and 
 }\Lambda_1\leq -\frac{1}{B}\log(||e^G-e^H||_{L^1(\omega_M^n)}).
\]\end{proof}

By the iteration lemma \ref{DeGiorgiiterationlemma} we have $\phi(s)=0$ for \begin{equation}\label{sinftystability}
    s\geq s_\infty=\frac{2C_2e^{C_3\Lambda_1}}{1-2^{-\delta_0}}\phi(s_0)^{\delta_0}+s_0.
\end{equation}
    
On the other hand, \begin{equation}
     \begin{aligned}\label{phis0stability}
    \phi(s_0)=&\int_{E_{s_0}}f^{\frac{n}{k}}\omega_M^n\\
    \leq &\left(\int_Mf^{q}\omega_M^n\right)^{1-b}\left(\int_{E_{s_0}}f\omega_M^n\right)^{b}\\
    \leq &Ce^{q(1-b)\Lambda_1}\frac{||e^G-e^H||_{L^1(\omega_M^n)}}{r^3}\\
    \leq &Ce^{C\Lambda_1}||e^G-e^H||_{L^1(\omega_M^n)}.
    \end{aligned}
\end{equation}
   
 And 
 $s_0$ can be estimated as \begin{equation}
     s_0\leq C\int_{\Lambda_1}^\infty\Phi^{-\frac{1}{n}}(x)dx+2\beta r\leq C(\int_{\Lambda_1}^\infty\Phi^{-\frac{1}{n}}(x)dx+\Phi^{-\frac{1}{n}}(\Lambda_1))
 \end{equation}
 Combine \eqref{sinftystability},\eqref{phis0stability},\eqref{phis0stability}, we get for some $\eta>0$, \[
 s_\infty\leq C\int_{-\frac{1}{C}\log||e^G-e^H||_{L^1(\omega_M^n)}}^\infty \Phi^{-\frac{1}{n}}(x)dx+C\Phi^{-\frac{1}{n}}(-\frac{1}{C}\log(||e^G-e^H||_{L^1(\omega_M^n)})) +C||e^G-e^H||^\eta_{L^1(\omega_M^n)}.
 \]
 It's clear that $s_\infty$ is a constant which congverges to $0$ as $||e^G-e^H||_{L^1(\omega_M^n)}\rightarrow 0.$ The proof of the theorem is completed.

\end{proof}

\begin{lemma}\label{controlAs}
$\int_{E_{s_0}} f\omega^n\leq \frac{c_1||e^G-e^H||_{L^1(\omega_M^n)}}{r^3}$ for $s_0=\sup_M(a\psi)+3\beta r$, and constant $c_1$ is uniform.
\end{lemma}
\begin{proof}
Assume that $||e^G-e^H||_{L^1(\omega_M^n)}$ is small, say it's less than $\frac{1}{5}$.

Switching $u,v$ if necessary, we assume that $\int_{\{v\leq u\}}e^H\omega_M^n\leq \frac{1}{2}$.
Consider $S:=\{e^G\leq (1-r^2)e^H\}$, and note the fact $f\leq e^H$, we have $$r^2\int_Sf\omega^n\leq r^2\int_Se^H\omega_M^n\leq \int_S(e^H-e^G)\omega_M^n\leq ||e^G-e^H||_{L^1(\omega_M^n)}.$$
So $\int_Sf\omega_M^n\leq r^{-2}||e^G-e^H||_{L^1(\omega_M^n)}.$

On the other hand, on $M\setminus S$, we have $e^F\geq (1-r^2)e^F$. 
Consider the auxiliary function $\rho_j$ which is the solutions to equations\[
    (\omega+\ddc\rho_j)^k\wedge\omega_M^n=c_t(\frac{3}{2}\chi_je^{H}+c_j(1-\chi_j)e^H)\omega_M^n, \quad \sup\rho_j=0,
\]
where $\chi_j$ is smooth positive and approximates $\chi_{\{v\leq u\}}$ in $L^\infty$, with $\chi_j\geq 1$ on $\{v\leq u\}$, and $c_j$ is a constant such that $\int_M(\frac{3}{2}\chi_je^H+c_j(1-\chi_j)e^H)\omega_M^n=1$. We can assume that $0\leq \chi_j\leq 1$. 
Since $\int_{\{v\leq u\}}(e^H+e^G)\omega_M^n\leq 1$, we have  \[
    \int_{\{v\leq u\}}e^H\omega_M^n= \frac{1}{2}\left(\int_{\{v\leq u\}}(e^{H}+e^G)\omega_M^n+\int_{\{v\leq u\}}(e^H-e^G)\omega_M^n\right)\leq \frac{3}{5}.
\]
Therefore by the volume-normalization, the constant $c_j$ satisfies \[
    \frac{1}{20}\leq c_j\leq 3.
\]
Hence $\cN_{\Phi'}(\log(\frac{3}{2}\chi_je^{\frac{n}{n}H}+c_j(1-\chi_j)e^{\frac{n}{k}H}))<100K$ when we choose $\Phi'=\Phi(x-100)$. By theorem \eqref{degenerateuniformestimate}, we have $0\leq -\rho_j\leq C$ uniformly, increasing $\beta$ if necessary, and we can  assume that $|u|\leq \beta, |v|\leq \beta, |\rho_j|\leq \beta.$ 
It leads to the following inclusions\[
    E_s\subset \{-v+(1-r)u+\frac{3}{4}r\rho_j-\frac{7}{4}\beta r\geq 0\}\subset \{-v+u\geq 0\},
\]for $s\geq s_0$.
On $\{-v+u\geq 0\}\setminus S$, we have $e^G\geq(1-r^2)e^H$, so
\[\begin{aligned}
((1-r)\omega_u+\frac{3}{4}r\omega_{\rho_j})^k\wedge\omega_M^{n-k}=c_t&\sum_{i=0}^{k}\tbinom{k}{i}(1-r)^{i}(\frac{3}{4}r)^{k-i}\omega_u^i\wedge\omega_{\rho_j}^{k-i}\wedge\omega_M^{n-k}\\
\geq &\sum_{i=0}^{k}\tbinom{k}{i}(1-r)^{i}(\frac{3}{4}r)^{k-i}(\omega_u^k\wedge\omega_M^{n-k})^{\frac{i}{k}}\wedge(\omega_{\rho_j}^k\wedge\omega_M^n)^{\frac{k-i}{k}}\\
\geq &c_t\sum_{i=0}^{k}\tbinom{k}{i}(1-r)^{i}(\frac{9}{8}r)^{k-i}(1-r^2)^{\frac{i}{k}}e^H\omega_M^n\\
=&c_t\left((1-r)(1-r^2)^{\frac{1}{k}}+\frac{9}{8}r\right)^ke^H\omega_M^n\\
\geq &c_t(1+c_0r)e^H\omega_M^n,
\end{aligned}
\] for some positive constant $c_0>0$.

So we have\[
    \begin{aligned}
    c_t(1+c_0r)\int_{\{-v+(1-r)u+\frac{3}{4}r\rho_j-\frac{7}{4}\beta r\geq 0\}\setminus S}e^H\omega_M^n\leq &\int_{\{-v+(1-r)u+\frac{3}{4}r\rho_j-\frac{7}{4}\beta r\geq 0\}\setminus S}\omega_{(1-r)u+\frac{3}{4}r\rho_j}^k\wedge\omega_M^{n-k}\\
    \leq &\int_{\{-v+(1-r)u+\frac{3}{4}r\rho_j-\frac{7}{4}\beta r\geq 0\}}\omega_v^{k}\wedge\omega_M^{n-k}\\
    =&c_t\int_{\{-v+(1-r)u+\frac{3}{4}r\rho_j-\frac{7}{4}\beta r\geq 0\}\setminus S}e^H\omega_M^n+c_t\int_{S}e^H\omega_M^n,
    \end{aligned}
\]which implies\[
    \int_{\{-v+(1-r)u+\frac{3}{4}r\rho_j-\frac{1}{4}\beta r\geq 0\}\setminus S}e^H\omega_M^n\leq \frac{1}{c_0r}\int_Se^H\omega_M^n\leq \frac{||e^H-e^G||_{L^1(\omega_M^n)}}{c_0r^3}.
\]
In conclusion, we get the desired estimate \[
    \int_{E_s}f\omega_M^n\leq\int_{E_s}e^H\omega_M^n\leq \frac{c_1||e^F-e^G||_{L^1(\omega_M^n)}}{r^3},
\]for some uniform constant $c_1$.

\end{proof}

\section{Degenerated class}
Let $(M,\omega_M)$ be a compact K\"ahler manifold, $\omega_M$ is a fixed K\"ahler metric with $\int_M\omega_M^n=1$. Let $\omega$ be a closed $(1,1)$-form on $X$ and assume that the class $[\omega]$ is nef. 
Take $\omega_t=\omega+t\omega_M$, $t\in (0,1]$. Then $[\omega_t]$ is a K\"ahler class though $\omega_t$ may not be positive. 
We define the upper envelop of $\omega_t$-psh functions to be \[
\mathcal{V}_t:=\sup\{u\in C^\infty(M,\RR)|\omega_t+\ddc u\geq 0, \text{ and }u\leq 0\}.
\]Note that $\mathcal{V}_t=0$ if the form $\omega_t$ is semipositive. 
Similary, define the upper envelop of form $\omega_t$ with respect the $\Gamma_k$-cone by \[
\cV_{t,k}:=\sup\{u\in C^\infty(M,\RR)|\lambda[\omega_t+\ddc u]\in \Gamma_k, u\leq 0\}.
\]
Recall that $\lambda[\alpha]$ means the eigenvalues of real $(1,1)$-form $\alpha$ with respect to $\omega_M$, and $\Gamma_k:=\{\lambda\in \RR^n|\sigma_{i}(\lambda)\geq 0, i=1,\cdots, k\}.$
Let $\Phi(x):\mathbb{R}\rightarrow \mathbb{R}^+$ be an increasing function and $\Phi(x)\rightarrow \infty$ as $x\rightarrow \infty$. Given $\Phi,F$, we denote $\cN_\Phi(F)=\int_M\Phi(F)e^F\omega_M^n$.
\begin{theorem}\label{degenerateuniformestimate}
Consider the family of Monge-Amp\`ere equations\[
    (\omega_t+\ddc\varphi_t)^k\wedge\omega_M^{n-k}=c_{t}e^{F_t}\omega_M^n, \quad \lambda[\omega_t+\ddc\varphi_t]\in \Gamma_k, \quad \sup\varphi_t=0,
\]where $c_t={[\omega_t]^k[\omega_M]^{n-k}}$ and $F$ is normalized to satisfy $\int_Me^{F_t}\omega_M^n=\int_M\omega_M^n=1$. Assume that $\int_0^\infty\frac{ds}{\Phi^{1/n}(s)}<\infty.$Then there exists a constant $C=C(M,\omega_M,n,\omega,c_t^{\frac{n}{k}}/[\omega_t^n],\cN_\Phi(F),\Phi)$ such that for all $t\in(0,1],$
\[
    -\varphi_t+\mathcal{V}_{t,k}\leq C.
\]
\end{theorem}
In the proof, we need the following lemma due to Berman when $k=n$ and Guo-Phong-Tong for other $k$.  
\begin{lemma}[Berman \cite{MR3936074}, Guo-Phong-Tong \cite{MR4713113}]\label{smoothapproxofupperenvelope}
The solution $u_\beta$ to the complex Monge-Amp\`ere equation\[
    (\omega_t+\ddc u_\beta)^{k}\wedge\omega_M^{n-k}=c_te^{\beta u_\beta}\omega_M^n,\quad  \lambda[\omega_t+\ddc u_\beta]\in \Gamma_k,
\]converges uniformly to $\mathcal{V}_{t,k}$ as $\beta\rightarrow \infty$.
\end{lemma}
\begin{proof}
For any given $q>1,a\in(0,1)$, 
we define $h(s):\mathbb{R}^+\rightarrow \mathbb{R}$ by \[
    h(s)=\int_1^s\frac{1}{a}\frac{q}{\alpha}\frac{c_t^{\frac{1}{k}}}{[\omega_t^n]^\frac{1}{n}}\cN_\Phi(F)^{1/n}\frac{dt}{\Phi^{1/n}(t)},
\]
where $\alpha$ is an uniform upper bound for the $\alpha$-invariant for $\omega_t$.

Consider the first auxiliary equation\[
    (\omega_t+\ddc \psi_1)^n=\frac{[\omega_t]^n\Phi(F_t)e^{F_t}}{\cN_\Phi(F)}\omega_M^n,\quad \sup\psi_1=0.
\]
Let $\psi=u_\beta-h(-\frac{\alpha}{q}(\psi_1-u_\beta-1)+\Lambda_1)$, with $\Lambda_1$ chosen to satisfy $\frac{\alpha}{q}h'(\Lambda_1)=1$

Since $u_\beta$ converges to $\mathcal{V}_{t,k}$ uniformly by lemma \ref{smoothapproxofupperenvelope} and by definition of $\mathcal{V}_{t,k}$ we have $\psi_1\leq \mathcal{V}_{t,k}$, we can take $\beta$ large enough to ensure that $\psi_1-u_\beta-1\leq 0$. 
With such $\Lambda_1$, one has $\frac{\alpha}{q}h'(-\frac{\alpha}{q}(\psi_1-u_\beta+1)+\Lambda_1)\leq 1$.

Hence, we have \[\begin{aligned}
 \omega_t+\ddc\psi&=\omega_t+\ddc u_\beta+\frac{\alpha}{q}h'\ddc(\psi_1-u_\beta)\\
 &+(\frac{\alpha}{q})^{2}(-h'')\sqrt{-1}\partial(\psi_1-u_\beta)\wedge\bar{\partial}(\psi_1-u_\beta)\\
 & =(1-\frac{\alpha}{q}h')(\omega_t+\ddc u_\beta)+\frac{\alpha}{q}(h')(\omega_t+\ddc\psi_1)\\
 &+(\frac{\alpha}{q})^{2}(-h'')\sqrt{-1}\partial(\psi_1-u_\beta)\wedge\bar{\partial}(\psi_1-u_\beta).
\end{aligned}
\]
From $\lambda[\omega_t+\ddc u_\beta]\in \Gamma_k$ and $
    \frac{\alpha}{q}h'(-\frac{\alpha}{q}(\psi_1-u_\beta+1)+\Lambda_1)\leq 1$,
 we have \[\begin{aligned}
     (\omega_t+\ddc\psi)^k\wedge\omega_M^{n-k}\geq&(\frac{\alpha}{q}h'(\omega_t+\ddc\psi_1))^k\wedge\omega_M^{n-k} \\
     \geq &(\frac{\alpha}{q}h')^k\left(\frac{(\omega_t+\ddc\psi_1)^n}{\omega_M^n}\right)^{\frac{k}{n}}\omega_M^n\\
     =&(\frac{\alpha}{q}h')^{k}\frac{[\omega_t^n]^k\Phi(F_t)^\frac{k}{n}e^F}{\cN_\Phi(F_t)^\frac{k}{n}}.
 \end{aligned}
 \]
Therefore by a similar argument in section 3, we have \begin{equation}\label{degeneratecomparison}
c_te^F\omega_M^n\leq a^k(\omega_t+\ddc\psi)^k\wedge\omega_M^{n-k}+c_tf\omega_M^n
    \end{equation}
with $f=\exp\{(-\frac{\alpha}{q}(\psi_1-u_\beta-1)+\Lambda_1)\}$. Then one has $||f||_{L^q(\omega_M^n)}\leq Ce^{\Lambda_1}$.

Next, we consider our second auxiliary function $\Psi_2$, which is the solution to complex Monge-Amp\`ere equation \[
    (\omega_t+\ddc\psi_2)^{n}=\frac{[\omega_t^n]\eta_j(-\varphi_t+(1-a)u_\beta+a\psi-s)f^{\frac{n}{k}}}{A_{s,j}}\omega_M^n,
\]where \[
    A_{s,j}=\int_M\eta_j(-\varphi_t+(1-a)u_\beta+a\psi-s)f^{\frac{n}{k}}\omega_M^n,
\] and $\eta_j(x)$ is a sequence of smooth positive functions approximating $x_+$, the positive part of $x$, from above. 
Now, consider the function \[
    \Psi:=-\varphi_t+(1-a)u_{\beta}+a\psi-s-\epsilon_2(-\psi_2+u_\beta+1+\Lambda_2)^{\frac{n}{n+1}},
\]where \[
    \epsilon_2=\left(\frac{c_t^{1/k}}{[\omega_t^n]^{1/n}}\right)^{\frac{n}{n+1}}\left(\frac{n+1}{n}\right)^{\frac{n}{n+1}}A_{s,j}^{\frac{!}{n+1}},\quad \Lambda_2=\frac{n}{n+1}\left(\frac{c_t^{1/k}}{[\omega_t^n]^{1/n}}\right)^{n}\frac{1}{(1-a)^{n+1}}A_{s,j}.
\]
Then, at the point where $\Psi$ attains its maximum, we have \[
    \begin{aligned}
    0\geq &-\ddc\varphi_t+(1-a)\ddc u_\beta+a\ddc\psi\\
    &+\epsilon_{2}\frac{n}{n+1}(-\psi_2+u_\beta+1+\Lambda_2)^{-\frac{1}{n+1}}(\ddc\psi_2-\ddc u_\beta)\\
    =&-(\omega_t+\ddc\varphi_t)+a(\omega_t+\ddc\psi)+\epsilon_2\frac{n}{n+1}(-\psi_2+u_\beta+1+\Lambda_2)^{-\frac{1}{n+1}}(\omega_t+\ddc\psi_2)\\
    &+\left((1-a)-\epsilon_2\frac{n}{n+1}(-\psi_{2}+u_\beta+1+\Lambda_2)^{-\frac{1}{n+1}}\right)(\omega_t+\ddc u_\beta)
    \end{aligned}
\] 
From our choice of $\epsilon,\Lambda_2$, the coefficient of the last term is positive, so one gets \[\begin{aligned}
      (\omega_t+\ddc\varphi_t)^k\wedge \omega_M^{n-k}\geq &a^k(\omega_t+\ddc\psi)^k\wedge\omega_M^{n-k}\\
      &+\left(\epsilon_2\frac{n}{n+1}(-\psi_2+u_\beta+1+\Lambda_2)^{-\frac{1}{n+1}}\right)^k(\omega_t+\ddc\psi_2)^k\wedge\omega_M^{n-k}\\
      \geq &a^k(\omega_t+\ddc\psi)^k\wedge\omega_M^{n-k}\\
      &+\left(\epsilon_2\frac{n}{n+1}(-\psi_2+u_\beta+1+\Lambda_2)^{-\frac{1}{n+1}}\right)^k\left(\frac{(\omega_t+\ddc\psi_2)^n}{\omega_M^n}\right)^{\frac{k}{n}}\omega_M^n
\end{aligned}  
\]
Together with equation \eqref{degeneratecomparison}, and definition equation of $\psi_2$, the choice of $\epsilon,\Lambda_2$, and $\eta_j(x)\geq x$, we get $\Psi\leq 0.$
The iteration process is similar as before to show the uniform estimate \[
-\varphi+(1-a)u_\beta+a\psi\leq C.
\]
For example, we can estimate the term $\Lambda_2$ uniformly as \[\begin{aligned}
    \Lambda_2\leq &C\int_{M}(-\varphi_t+(1-a)u_\beta+a\psi)_+f^{\frac{n}{k}\omega_M^n}\\
    \leq &C\int_M(-\varphi_t+u_{\beta})_+f^{\frac{n}{k}}\omega_M^n\\
    \leq &C\int_M(-\varphi_t+1)_+f^{\frac{n}{k}}\omega_M^n\\
    \leq &C||-\varphi_t+1||_{L^{p}(\omega_M^n}||f^{\frac{n}{k}}||_{L^p*(\omega_M^n)}\leq C,
\end{aligned} 
\]when we choose $p$ close to 1, and $q>\frac{np^*}{k}$ for $p*^{-1}+p^{-1}=1$. $L^p$-norm of $\varphi_t$ is uniformly bounded by lemma 8 in \cite{MR4593734} for any $p\in [1,\frac{n}{n-1})$.  Let $\beta\rightarrow \infty$, and note that $h(s)$ is a bounded function since we assume $\int_1^\infty\Phi^{-\frac{1}{n}}<\infty$, we get $-\varphi-\cV_{t,k}\leq C$. The proof is completed.
\end{proof}

In the proof, we have used the following lemma (proposition 2.1 in \cite{MR2172278}) several times:
\begin{lemma}
    Let $\alpha_1,\cdots, \alpha_k$ be real $(1,1)$-forms with $\lambda[\alpha_i]\in \Gamma_k$ for all $i=1,\cdots,k$. Then we have \[
    \alpha_1\wedge\cdots\wedge\alpha_k\wedge\omega_M^{n-k}\geq 0.
    \]
\end{lemma}

\section{Geometric estimate}

In this section, we consider the following subset of the K\"ahler cone on compact K\"ahler manifold $(M,\omega_M)$, where $\omega_M$ is a fixed K\"ahler metric with normalization $[\omega_M]^n=1$,
\[
    \mathcal{W}:=\mathcal{W}(n,A,\Phi,K):=\{\omega \text{ is a K\"ahler metric }:[\omega][\omega_M]^{n-1}\leq A, \mathcal{N}_{\Phi}(\omega)\leq K\},
\]
where \[
    F_\omega:=\log(\frac{\omega^n/[\omega]^n}{\omega_M^n/[\omega_M]^n}), \quad
    \mathcal{N}_{\Phi}(\omega):=\frac{1}{[\omega]^n}\int_{M}\Phi(F_\omega)\omega^n=\int_M\Phi(F_\omega)e^{F_\omega}\omega_M^n,
\]
We use the following definition of  Green's function $G(x,y)$ associated to a Riemmanian manifold $(M,g)$ by \[
    \Delta_gG(x,\cdot)=-\delta_x(\cdot)+\frac{1}{Vol(M,g)}, \quad\text{ with }\int_MG(x,\cdot)\omega^n=0.
\]
Then one has the following estimates on Green's function:
\begin{theorem}
For any given $p>2n$,let $\Phi(x)=|x|^n(\log |x|)^{p}$. For any $r\in (1,\frac{p}{n}-2), $
there exists a constant $C=C(M,\omega_M,n,A,K,p,r)$, such that for any $\omega\in \mathcal{W}$, we have \[
    -\inf G(x,\cdot)[\omega]^n+\int_M|G(x,\cdot)|\log^r([\omega]^nG-\inf([\omega]^nG)+2)\omega^n+\int_M|\nabla G(x,\cdot)|\omega^n\leq C
    \]
\end{theorem}
\begin{remark}
\begin{enumerate}
\item Actually, we only use the explicit expression of $\Phi$ in the final step to show the $L^1$-bound of $\nabla G$. 
For the lower bound of Green's function $-G$ and uniform $L^1$-norm of $G$, one needs only require that $\int_0^\infty\Phi(t)^{-\frac{1}{n}}dt< \infty.$
\item Such type estimate are firstly proved by Guo-Phong-Song-Sturm in \cite{guo2022diameter} with additional assumption on a positive lower bound of the volume form on a large part of the manifold. In many applications the assumption can be ensured by the lower bound of the scalar curvature. Remarkably, this assumption is removed recently in \cite{guo2024diameter2} and \cite{guedj2024diameter}.  
\end{enumerate}{}
\end{remark}
Note that for any K\"ahler metric $\omega\in \mathcal{W}$, $\omega$ has bounded mass as current, the following lemma (Proposition 3.1 in \cite{guo2022diameter}) enables us to choose a good representative $(1,1)$-form in class $[\omega]$ with uniform control on its $C^k$-norm. 
\begin{lemma}\label{choiceofrepresentative}Let $(M^n,\omega_M)$ be a compact K\"ahler manifold equipped with a K\"ahler metric $\omega_M$. For any $k\geq 0$, and a nonnegative class $\beta\in H^{1,1}(M,\mathbb{R})$, there exists a smooth representative $\theta\in \beta$ such that  \[
    ||\theta||_{C^k(\omega_M)}\leq C(M,\omega_M,k,[\beta]\cdot[\omega_M]^{n-1}).
\]
\end{lemma}
In the following of the section, for any $\omega\in \mathcal{W}$, we choose $\omega=\theta+\ddc\varphi$ with $\sup\varphi=0$, and $||\theta||_{C^k(\omega_M)}\leq C$ uniformly.

With the help of the relative comparison method, we can also prove the following lemma
\begin{lemma}\label{onesideuniform}Assume $\omega\in \mathcal{W}$ and $v\in L^1(\omega^n)$. Let $\Omega_s:=\{v\geq s\}$. If \[v\in C^2(\Omega_{-1},\mathbb{R}), \quad \Delta v\geq -1 \text{ in }\Omega_0.\] Then \[
    \sup v\leq C(1+\int_M|v|f_\omega\omega_M^n),
\] for some constant $C$ depending on the parameters in $\mathcal{W}$, and $f_\omega$ is defined in \eqref{definef}
\end{lemma}
\begin{proof}
It's sufficient to show that $\sup v\leq C$ under an additional assumption $||v||_{L^1(f\omega_M^n)}=1$, where the constant $C$ depends on the parameters in $\cW$.

By lemma \ref{choiceofrepresentative}, there exists a smooth $\theta\in [\omega]$ with uniformly bounded $||\theta||_{C^k(\omega_M)}$. As in the previous section, we take $\psi_1$ to be the solution to the equation \[
    (\theta+\ddc\psi_1)=[\omega]^n\frac{\Phi(F)}{\mathcal{N}_\Phi(F)}e^F\omega_M^n,\quad \sup\psi_1=0,
\] where $F=F_\omega$, and for any $a\in(0,1),q>1$ define $h$ as \[
    h(s)=-\int_s^\infty\frac{\alpha \mathcal{N}_{\Phi}(\omega)^{\frac{1}{n}}dt}{aq\Phi(t)^{\frac{1}{n}}}.
\]
In the above, we choose $\alpha$ to be an upper bound for the $\alpha$-invariant for $\theta$ which is uniformly bounded as $||\theta||_{C^k(\omega_M)}$ is uniformly bounded. 

We define $\cV_\theta=\sup\{u|\theta+\ddc u\geq 0, u\leq 0\}$ and $u_\beta$ is the smooth approximation of $\cV_\theta$ constructed by lemma \ref{smoothapproxofupperenvelope}.
Then take $\psi=u_\beta-h(-\frac{\alpha}{q}(\psi_1-u_\beta-1)+\Lambda_1),$ and $
f_\beta=\min\{\exp\{-\frac{\alpha}{q}(\psi_1-u_\beta-1)+\Lambda_1\},e^F\}.$ Then $f_\beta$ converges uniformly to \begin{equation}\label{definef}
    f_\omega:=\min\{\exp\{-\frac{\alpha}{q}(\psi_1-\cV_\theta-1)+\Lambda_1\},e^F\}.
\end{equation} When there is no ambiguity, we omit the subscript $\omega$ in $f_\omega$.

 As in the previous section, we have \begin{equation}\label{compareinequality}
    \omega^n=(\theta+\ddc\varphi)^n=[\omega^n]e^F\omega_M^n\leq a^n(\theta+\ddc\psi)^n+[\omega]^nf_\beta\omega_M^n.
\end{equation}
Now, we consider the auxiliary function $\psi_{2,j}$ which is the solution to the equation \[
    (\theta+\ddc\psi_{2,j})^n=[\omega]^n\frac{\eta_j(v-\varphi+(1-r)u_\beta+r\psi-s)}{A_{s,j}}f_\beta\omega_M^n,\quad \sup\psi_{2,j}=0,
\] where \[
    \tilde{\Omega}_s=\{v-\varphi+(1-r)u_\beta+r\psi-s>0\},\quad A_{s,j}=\int_M\eta_j(v-\varphi+(1-r)u_\beta+r\psi-s)f_\beta\omega_M^n.
\]
By theorem theorem \ref{degenerateuniformestimate}, there is a uniform constant $C$ such that for $s\geq C$, we have $\tilde{\Omega}_s\subset \Omega_{0}$.
Next, we consider the function \[
    \Psi=v-\varphi+(1-r)u_\beta+r\psi-s-\epsilon_2(-\psi_{2,j}+u_\beta+1+\Lambda_2)^{\frac{n}{n+1}},
\]
where
\[  r=\frac{n+1}{n}a, \quad 
    \epsilon_2=\left(\frac{n+1}{n}\right)^{\frac{2n}{n+1}}A_{s,j}^{\frac{1}{n+1}},\quad \Lambda_2=\left(\frac{n+1}{n}\right)^{n-1}\frac{1}{(1-r)^{n+1}}A_{s,j}.
\]
By maximum principle and inequality \eqref{compareinequality}, we have $\Psi\leq 0$.

As mentioned in the beginning, we can assume that $||v||_{L^1(f\omega_M^n)}\leq1$. So we have $A_{s,j}\leq C$ for $\beta$ large enough and a uniform constant $C$. The remaining iteration argument is the same as before and we arrive at \[
    v-\varphi+(1-r)u_\beta+r\psi\leq C.
\]
Let $\beta\rightarrow \infty$, we get \[
    v-\varphi+\cV_\theta-rh(-\frac{\alpha}{q}\psi_1+\Lambda_1)\leq C.
\]
Note that from the definition of $\cV_\theta,$ and $h$,  we have $-\varphi+\cV_\theta\geq 0$ and $h\leq 0$. Hence we arrive at $v\leq C$ which closes the proof. 
\end{proof}

\begin{lemma}
\label{uniformestimate}For any $\omega\in \cW$, and any $C^2$-function on $M$ with \[
|\Delta_\omega v|\leq 1, \quad \text{and }\int_Mv\omega^n=0,
\] we have \[
    \frac{1}{[\omega]^n}\int_M|v|\omega^n\leq C
\]
for some constant $C$ depending on the parameters in $\cW$.
\end{lemma}
\begin{proof}Note that our choice of $f_\omega$ satisfies $f_\omega\omega_M^n\leq \frac{1}{[\omega]^n}\omega^n.$ So for $\omega\in \cW$, the proof of Propsition 2.1 in \cite{guo2024diameter2} works. 
\end{proof}
Now we can begin our proof of the estimate of Green's function associated to $\omega$. The first two items about the uniform $L^1$-norm of $G$ and lower bound of $[\omega]^nG$ are proved in \cite{guo2022diameter}, for the reader's convenience we include it here.
\begin{enumerate} 
\item $||G||_{L^1(\omega^n)}\leq C$:

Let $\chi_k(y)$ be a sequence of smooth functions converge to $\chi_{\{G(x,y)\leq0\}}$ uniformly, and solve the equation\[
    \Delta_{\omega}u_k=\chi_k-\frac{1}{[\omega]^n}\int_M\chi_k\omega^n,\quad \int_Mu_k\omega^n=0.
\]
Without loss of generality, we can assume that $|\chi_k|\leq 2$ for all $k$. Then we have $|\chi_k-\frac{1}{[\omega]^n}\int_M\chi_k\omega^n|\leq 4$.{}

By lemma \ref{uniformestimate} and lemma \ref{onesideuniform}, we get $
    |u_k|\leq C $ for some uniform constant $C$. By Green's formula, we have\[ \begin{aligned}
       u_k(y)=&\frac{1}{[\omega]^n}\int_Mu_k\omega^n+\int_M-G(y,z)\Delta_{\omega}u_k\omega^n\\
       =&\frac{1}{[\omega]^n}\int_Mu_k\omega^n+\int_{M}-G(y,z)\chi_k(y)\omega^n\\
       \rightarrow& \int_{\{G(x,\cdot)\leq 0\}}|G(x,\cdot)|\omega^n
    \end{aligned}
 \]
 Hence $\int_{\{G(x,\cdot)\leq 0\}}|G(x,\cdot)|\omega^n\leq C$ for some uniform constant $C$. According to the normalization of $G$, we get the desired uniform estimate for $||G(x,\cdot)||_{L^1(\omega^n)}$.

\item $-[\omega]^n\inf G(x,\cdot)\leq C$.

 Let $v(\cdot)=-[\omega]^nG(x,\cdot)$, then $\Delta v=-1$ on $\{v\geq -1\}$. By lemme \ref{onesideuniform}, $\sup v\leq C(1+\int_M|v|f\omega_M^n)=C(1+[\omega]^n\int_M|G|f\omega_M^n)\leq C$.

Now, let's choose another normalization, \[
    \cG(x,y)=G(x,y)-\inf_{x,y} G(x,y)+\frac{1}{[\omega]^n}.
\]
So, $\cG>0$ and $\int_M\cG\omega^n\leq C$ for some constant $C$ uniformly. 

\item $\int_M\cG(\log([\omega]^n\cG))^{r}\omega^n\leq C$ for some $r>1$.

Let $\cG_k(x,y):=\min\{\cG(x,y),k\}$ and approximate $\cG_k(x,y)$ by a sequence of smooth functions $\cG_{k,l}(x,y)$ uniformly. For any $r>1$, define $H_{k,l}$ as \[
H_{k,l}:=\frac{1}{2}\log^{r}([\omega]^n\cG_{k,l}+1)\]
Consider the following two auxiliary equations \[
    (\theta+\ddc\psi_{3,k,l})^n=[\omega]^n\frac{(H_{k,l}^n+1)e^{F}\omega_M^n}{B_{k,l}},\quad \sup\psi_{3,k,l}=0,
\]
where $B_{k,l}:=\int_M(H_{k,l}^n+1)e^{F}\omega_M^n$, and 
\[\Delta_\omega u_{k,l}=-H_{k,l}+\int_MH_{k,l}e^F\omega_M^n, \quad \int_Mu_{k,l}\omega^n=0.\]

Since $\int_M\cG\omega^n\leq C$ uniformly, and note the fact $\log(1+x)\leq C_\delta x^\delta$ for any $\delta\in(0,1), x>0$, we have  \[
    \int_MH_{k,l}^ne^F\omega_M^n=\frac{1}{2^n}\int_M\log^{nr}([\omega ]^n\cG_{k,l}+1)e^F\omega_M^n\leq C(r,n)\int_M[\omega]^n\cG_{k,l}e^F\omega_M^n\leq C.
\]
On the other hand,  $1=\int_Me^F\omega_M^n\leq B_{k,l}.$

Next we estimate the $L^1(\log L)^n\log^{p'}(\log L)$-norm of $(H_{k,l}^n+1)e^FB_{k,l}^{-1}$ for some $p'>n$.

By elementary inequality $(x+y)^n\leq 2^nx^n+2^ny^n$ and $\log(x+y+z)\leq \log(1+x)+\log(1+y)+\log(1+z)$ for any $x,y,z>0$, we have  \[
\begin{aligned}
 &\int_M \frac{(H_{k,l}^n+1)e^F}{B_{k,l}}|\log(H_{k,l}^n+1)+F-\log B_{k,l}|^n|\log^{p'}(|\log(H_{k,l}^n+1)+F-\log B_{k,l}|)\omega_M^n\\
 \leq &\int_MC\frac{(H_{k,l}^n+1)e^F}{B_{k,l}}\left(\log^n(H_{k,l}+1)+|F|^n+|\log B_{k,l}|^{n}\right)\\
 &\cdot\left(\log^{p'} (1+\log(H_{k,l}^n+1))+\log^{p'}(1+|F|)+\log^{p'}(1+|\log B_{k,l}|)\right)\omega_M^n\\
 = &I_1+I_2+I_3+I_4+I_5+I_6+I_7+I_8+I_9
\end{aligned}
\]
Expanding the bracket, we get 9 terms to control: For instance $I_1$,
 \[\begin{aligned}
I_1=&C\int_M\frac{(H_{k,l}^n+1)e^F}{B_{k,l}}\log^n(H_{k,l}+1)\log^{p'}(1+\log(H_{k,l}+1))e^F\omega_M^n\\
\leq &C\int_M(H_{k,l}^{n+\epsilon}+1)e^F\omega_M^n\leq C.
\end{aligned}\]
The nontrivial term involves $$\int_M(H^n_{k,l}+1)|F|^n\log^{p'}(|F|+1)e^F\omega_M^n.$$ We deal with such term in details. For $p>2n$, we choose $\delta>0$ small enough and $n<p'<p-n$ and $r=\frac{p-p'}{n}>1$. Then by lemma \ref{younginequality} with $\epsilon=1, t=p-p'$, we get \[\begin{aligned}
    &(H^n_{k,l}+1)|F|^n\log^{p'}(|F|+1)\\
    \leq & (H_{k,l}^n+1)\exp\{(H_{k,l}^n+1)^{\frac{1}{p-p'}}\}\\
    &+|F|^n\log^{p'}(|F|+1)\log^{p-p'}(|F|^n\log^{p'}(|F|+1)+1)\\
    \leq &C\exp\{2H_{k,l}^{\frac{n}{p-p'}}\}+C|F|^n\log^p(|F|+1)\\
    \leq &C\exp\{[\omega]^nG_{k,l}+1\}+C|F|^n\log^p(|F|+1)
\end{aligned}
\]Therefore \[
    \int_M\frac{(H_{k,l}^n+1)e^F}{B_{k,l}}|F|^n\log^{p'}(|F|+1)e^F\omega_M^n\leq C.
\]
It follows that the $L^1\log^n(L)\log^{p'}\log L$-norm of $(H_{k,l}^n+1)e^FB_{k,l}^{-1}$ is uniformly bounded from above. So by theorem \ref{degenerateuniformestimate}
\[
    -\psi_{k,l}+\cV_\theta\leq C,
\]for some uniform constant $C$.

Now consider the function $v_{k,l}=nB_{k,l}^{-\frac{1}{n}}u_{k,l}+\psi_{k,l}-\varphi-\frac{1}{[\omega]^n}\int_M(\psi_{k,l}-\varphi)\omega^n.$
\[
    \begin{aligned}
    \Delta_{\omega}v_{k,l}=&-nB_{k,l}^{-\frac{1}{n}}H_{k,l}+nB_{k,l}^{-\frac{1}{n}}\int_MH_{k,l}e^F\omega_M^n+tr_{\omega}(\theta_{\psi_{k,l}}-\theta_\varphi)\\
    \geq &-nB_{k,l}^{-\frac{1}{n}}H_{k,l}+n\left(\frac{\theta_{\psi_{k,l}}^n}{\omega^n}\right)^{\frac{1}{n}}-n\\
    \geq &-nB_{k,l}^{-\frac{1}{n}}H_{k,l}+nB_{k,l}^{-\frac{1}{n}}(H_{k,l}^n+1)^{\frac{1}{n}}-n\\
    \geq& -n
    \end{aligned}
\]
By Green's formula\[
    v_{k,l}=\frac{1}{[\omega]^n}\int_Mv_{k,l}\omega^n+\int_M-\cG\Delta_{\omega }v_{k,l}\omega^n\leq n\int_M\cG\omega^n\leq C.
\]

By theorem \ref{degenerateuniformestimate}, we also have \[
    -\varphi+\cV_\theta\leq C.
\]
Therefore $|\psi_{k,l}-\varphi|\leq C$ uniformly, and so is $\frac{1}{[\omega]^n}\int_{M}(\psi_{k,l}-\varphi)\omega^n.$

So, we get $u_{k,l}\leq C$ for a uniform constant $C$. Apply the Green's formula again to function $u_{k,l}$, we get \[
    C\geq u_{k,l}=\frac{1}{[\omega]^n}\int_Mu_{k,l}\omega^n+\int_M-\cG(-H_{k,l}+\frac{1}{[\omega]^n}\int_MH_{k,l}\omega^n)\omega^n,
\]which implies \[
    \int_M\cG H_{k,l}\omega^n\leq C.
\] 
Let $l,k\rightarrow \infty$, we get the desired estimate.

\begin{lemma}\label{younginequality}
For any $x>0,y>1$, and any $t>0,\epsilon>0$, one has \[
    xy\leq \epsilon y(\log y)^t+xe^{(\epsilon^{-1} x)^{\frac{1}{t}}}
\]
\end{lemma}
\begin{proof}
It's equivalent to show \[x-\epsilon(\log)^{t}\leq x\frac{e^{(\epsilon^{-1}x)^{\frac{1}{t}}}}{y}.\]
If $x\leq \epsilon(\log y)^{t}$, i.e.$y\geq \exp{(\epsilon^{-1}x)^{\frac{1}{t}}}$, the LHS is less than 0, and the inequality is trivial. Otherwise, if $y<\exp{(\epsilon^{-1}x)^{\frac{1}{t}}}$, we have RHS is greater than $x$, and the inequality is also valid. 
\end{proof}

\item Step 3. $\int_M|\nabla\cG|\omega^n\leq C$.
For the gradient estimate, we need a lemma
\begin{lemma}\label{L2gradientofGreen}
For any $\beta>0$, we have $$\int_M\frac{|\nabla\cG|^2}{\cG(\log([\omega]^n\cG))^{1+\beta}}\omega^n\leq \frac{1}{\beta}.$$
\end{lemma}
\begin{proof}
Consider the function $u=(\log([\omega]^n\cG))^{-\beta}(x,y)$. $u$ is continuous, smooth outside the diagonal of $M\times M$, equals $0$ on the diagonal and is bounded from above by $u\leq 1$ thanks to the normalization of $\cG$. 
Then one has \[
    \begin{aligned}
    u(x,y)=&\frac{1}{[\omega]^n}\int_Mu(x,\cdot)\omega^n-\int_M\cG(y,\cdot)\Delta u(x,\cdot)\omega^n\\
    = &\frac{1}{[\omega]^n}\int_Mu(x,\cdot)\omega^n-\beta\int_M\frac{|\nabla\cG|^2}{\cG(\log([\omega]^n\cG))^{1+\beta}}\omega^n
    \end{aligned}
\]Let $x\rightarrow y$, one gets the desired result. 
\end{proof}
Now the $L^1$-estimate for the gradient of Green functions follows as: 
\[
\int_M|\nabla\cG|\omega^n\leq \left(\int_M\frac{|\nabla\cG|^2}{\cG(\log([\omega]^n\cG))^{1+\beta}}\omega^n\right)^{1/2}\left(\int_{M}\cG(\log([\omega]^n\cG))^{1+\beta}\omega^n\right)^{1/2}\]
\end{enumerate}

\begin{corollary}Under the same assumption as in theorem \ref{greenfunction}, we have \begin{enumerate}
\item $diam(M,\omega)\leq C$,
\item the metric $\omega$ is noncollapsing:  there exists some $\delta>0$, such that for any $R\in(0,1]$,\[
    vol(B(x,R),\omega)e^{c_0R^{-\delta}}\geq C>0
\]
\end{enumerate}
\end{corollary}
\begin{proof}For the diameter bound: consider the distance function $d(x_0,x)$ and apply the Green's formula, we have \[
    d(x_0,x)=\frac{1}{[\omega^n]}\int_Md(x_0,\cdot)\omega^n+\int_M\nabla\cG(x,\cdot)\cdot\nabla d(x_0,\cdot)\omega^n.
\]
Let $x=x_0$, we get \[
\frac{1}{[\omega^n]}\int_Md(x_0,\cdot)\omega^n=-\int_M\nabla\cG(x,\cdot)\cdot\nabla d(x_0,\cdot)\omega^n.
\]Hence, by theorem \ref{greenfunction} we can estimate the distance function as \[
    d(x_0,x)\leq 2\sup_x\int_M|\nabla\cG(x,\cdot)|\cdot\omega^n\leq C.\]

For the volume noncollapsing: Given any geodesic ball $B(x_0,R)$ with respect to metric $\omega$, consider the smooth function $\eta$ with support in $B(x_0,R)$ such that $\eta(x)=1$ on $B(x,\frac{R}{2})$ and $\eta(x)=0$ on $\partial B(x_0,R)$. We can assume that $|\nabla\eta|\leq \frac{4}{R}$. Apply the Green's formula to $\eta$ we get\[
    \eta(x)=\frac{1}{[\omega^n]}\int_M\eta\omega^n+\int_M\nabla\cG\cdot\nabla\eta\omega^n.
\]For $y\in \partial B(x_0,R)$, we get \[
    \frac{1}{[\omega]^n}\eta\omega^n=\int_M\nabla\cG(y,\cdot)\cdot\nabla\eta(\cdot)\omega^n.
\]
Now for $x=x_0$, we have\[
    1=\frac{1}{[\omega^n]}\int_M\eta\omega^n+\int_M\nabla\cG(x_0,\cdot)\cdot\nabla\eta(\cdot)\omega^n\leq \frac{8}{R}\sup_x\int_{B(x_0,R)}|\nabla\cG(x,\cdot)|\omega^n.
\]
By lemma \ref{L2gradientofGreen}, we have \[
    \int_{B(x_0,R)}|\nabla\cG|\omega^n\leq \left(\int_M\frac{|\nabla\cG|^2}{\cG\log^r([\omega^n]\cG)}\omega^n\right)^{\frac{1}{2}}\left(\int_{B(x_0,R)}\cG\log^r([\omega^n]\cG)\omega^n\right)^{\frac{1}{2}}.
\]
We can use lemma \ref{younginequality} to estimate the last term as \[
\begin{aligned}\int_{B(x,R)}\cG\log^r([\omega^n]\cG)\omega^n\leq &\int_M\epsilon(\cG\log^r([\omega^n]\cG)\log^{t}([\omega^n]\cG\log^r([\omega^n]\cG)))\omega^n+\int_{B(x_0,R)}\exp\{\epsilon^{-\frac{1}{t}}\}\omega^n\\
\leq &\int_M\epsilon\cG\log^{r+2t}([\omega^n]\cG)\omega^n+\exp^{\epsilon^{-\frac{1}{t}}}\vol(B(x_0,R))\end{aligned}
\]
Choose $\epsilon=\frac{R}{C}$ for some large $C$ we get the desired estimate. 
\end{proof}

\begin{remark}
Recently, the authors in \cite{guo2023sobolev} showed the volume-noncollapsing results on singular K\"ahler space by extending the Sobolev inequality to that setting. one interesting question is whether their method can also be applied when assuming a weaker Orlicz norm on the volume form like here. One difficulty is that when assuming weaker Orlicz norm bound, the volume-noncollapsing is not Euclidean, i.e. not in the form of $\vol(B(x,r))\geq cr^\alpha$ for some $\alpha>0$ and it's hard to find a iteration scheme as in \cite{guo2023sobolev}.

The condition on $p>2n$ is very close to being sharp since there are examples where the estimates of Green's functions don't hold. The example below is a slight modification of the one in \cite{GGZ}

\begin{exam}
We consider the following local model on $B(0,1)\subset \mathbb{C}^n$(it's not hard to modify it to a mertric on a compact manifold): Let $t_\epsilon=-\log(|z|^2+\epsilon^2), \chi(t)=(\log t)^{-\alpha}$ for some $\alpha\in (0,1)$. We consider the metric $\omega_\epsilon=\ddc\chi(t_\epsilon)$. In the follopwing computation we suppress the subscript $\epsilon$ for convenience. Straightforward computation shows:\[
    \omega=\chi''(t)\frac{\sqrt{-1}\bar{z}_iz_{\bar j}dz^i\wedge d\bar{z}^j}{(|z|^2+\epsilon^2)^2}-\chi'(t)\frac{(\delta_{ij}|z|^2-\bar{z}_iz_{j})dz^i\wedge d\bar{z}^j}{(|z|^2+\epsilon^2)^2}=:\omega_r+\omega_s,
\]
where $\omega_r$ is the radial part and $\omega_s$ is the spherical part. Such metric are $U(n)$-invariant, therefore one can easily check that the unique geodesic connecting points $z,az$ with $a\in (0,1)\subset\mathbb{R}$ is just the radial line $\gamma(s)=sz+(1-s)az$.
Therefore, \[
    dist(z,az)=\int_0^1(1-a)\sqrt{\chi''(t(\gamma(s)))}ds=\int_{\log|z|a}^{\log|z|}\sqrt{\chi''(s)}ds
\]
By expression of $\chi,$ we have \[
    \chi'(t)=\frac{-\alpha}{t\log^{1+\alpha}t}, \quad \chi''(t)=\frac{(\alpha^2+\alpha)}{t^2\log^{\alpha+2}t}+\frac{\alpha}{t^2\log^{\alpha+1}t}.
\]
Hence the diameter is unbounded when $\epsilon\rightarrow 0$ since $\alpha\in (0,1).$

The volume form is \[
\begin{aligned}
   \omega^n=&\chi''(\chi')^{n-1}e^{nt}(\sqrt{-1})^{n}dz^1\wedge d\bar{z}^1\wedge\ldots \wedge dz^n\wedge d\bar{z}^n\\
   \sim &\frac{1}{e^{-nt}t^{n+1}\log^{(\alpha+1)n}t}dz^1\wedge d\bar{z}^1\wedge\ldots \wedge dz^n\wedge d\bar{z}^n.
\end{aligned}
\]
So the $L^1(\log L)^n(\log\log L)^p$-norm of the volume form is uniformly (when $\epsilon\rightarrow 0$) bounded when $(\alpha+1)n-p>1$.
\end{exam}
\end{remark}




\section{General setting }
In this section, we consider the more general Hessian equations on a compact K\"ahler manifold $(M,\omega_M)$:\begin{equation}\begin{aligned}
&g(\lambda[h_\varphi])=c_\omega e^{F},\\
&\lambda[h_\varphi]\in\Gamma, \sup\varphi=0,
\end{aligned}
\end{equation}
with $F$ normalized as $\int_Me^{nF}\omega_M^n=\int_M\omega_M^n$, and $g$ satisfies the structure conditions stated in the introduction. 




\begin{theorem}Let $(M,\omega_M)$ be a compact K\"ahler manifold with a fixed K\"ahler metric $\omega_M$. $\omega$ is another K\"ahler metric with $\omega\leq \kappa\omega_M$ for some constant $\kappa>0$, and $\varphi$ be a smooth solution to equation  \eqref{generalhessianequations}. We assume that the function $g$ satisfies the above structure conditions. For any increasing function $\Phi(x):\mathbb{R}\rightarrow \mathbb{R}$  with $\Phi(\infty)=\infty$, define a function $h(s)$ as $h(s)=\int_0^s\Phi(x)^{-\frac{1}{n}}dx$. Then, there exists a $\omega$-plurisubharmonic function $\psi$ with normalization $\sup\psi=0$, such that for any constant $\beta>0$, there exists constants $C_1,C_2,C_3$ depending on $M,n,\omega_M,\Phi, \frac{c_\omega^n}{V_\omega},\cN_\Phi(F),\beta,\kappa,c_0,\gamma_0$ s.t. \[
    -\varphi\leq C_1h(-\beta\psi+C_2)+C_3.
\]
where $V_\omega:=\int_M\omega^n, \cN_\Phi(F):=\int_M\Phi(F)e^{nF}\omega_M^n$ and $\psi$ is defined as \eqref{entropyauxiliary}.\end{theorem}
\begin{proof}
Take $\psi$ to be the solution of complex Monge-Amp\`ere equation\begin{equation}\label{entropyauxiliary}
    (\omega+\ddc\psi)^n=V_\omega\frac{\Phi(F)e^{nF}\omega_M^n}{\cN_\Phi(F)}, \quad \sup\psi=0.
\end{equation}
For any given $a\in (0,1),q>1$, define the function $h_1(s):\mathbb{R}\rightarrow \mathbb{R}$ as \[
    h_1(s):=\int_0^sc_\omega V_\omega^{-\frac{1}{n}}a^{-1}\alpha^{-1}q\cN_\Phi(F)^{\frac{1}{n}}\Phi(t)^{-\frac{1}{n}}dt,
\]where $\alpha $ is an uniform upper bound of $\alpha$-invariant for metric $\omega$. Since $\omega\leq \kappa\omega_M$, $\alpha$ can be chosen uniform with respect $\omega$. 
Now define the comparison function $\psi_1$ to be \begin{equation}\label{generalsettingdefinitionpsi}
    \psi_1:=-h_1(-\frac{\alpha}{q}\psi+\Lambda_1),
\end{equation} and $\Lambda_1$ is chosen to satisfy $\Phi(\Lambda_1)=a^{-n}\cN_\Phi(F)\frac{c_\omega^n}{V_\omega}.$
Similar calculation shows that \begin{equation}\label{generalsettingcomparison}
    g(\lambda[h_\varphi])=c_\omega e^{F}\leq a\left(\frac{\omega_{\psi_1}^n}{\omega_M^n}\right)^{\frac{1}{n}}+c_\omega f
\end{equation} where $f=\exp\{-\frac{\alpha}{q}\psi_1+\Lambda_1\}$.
Next, consider the auxiliary equation \[
    (\omega+\psi_{2,j})^n=V_\omega\frac{\eta_j(-\varphi+r\psi_1-s)f^n\omega_M^n}{A_{s,j}},\quad \sup \psi_{2,j}=0,
    \]
    and $A_{s,j}:=\int_M\eta_{j}(-\varphi+r\psi_1-s)f^n\omega_M^n$, $\Omega_s:=\{-\varphi+r\psi_1-s\geq0\}.$
    We claim that $\Psi:=(-\varphi+r\psi_1-s)-\epsilon_{2,j}(-\psi_{2,j}+\Lambda_{2,j})^{\frac{n}{n+1}}\leq 0$, where \[
    \epsilon_{2,j}=(c_0\gamma_0)^{\frac{n}{n+1}}\left(\frac{c_\omega^n}{V_\omega}\right)^{\frac{1}{n+1}}\left(\frac{n+1}{n^2}\right)^{\frac{n}{n+1}}A_{s,j}^{\frac{1}{n+1}}, \quad \Lambda_{2,j}=\gamma_0^nc_0^n\frac{c_\omega^n}{V_\omega}\frac{1}{n^{n-1}(n+1)}\frac{1}{(1-r)^{n+1}} .
    \]
    To see this, we can assume that $\Psi$ attains its maximum at $x_0\in\Omega_s$, then at $x_0$, we have \begin{equation}\label{geenralatmaximumineq}
    \begin{aligned}
    0\geq\ddc\Psi\geq &\ddc\varphi+r\ddc\psi_1+\epsilon_{2,j}\frac{n}{n+1}(-\psi_{2,j}+\Lambda_{2,j})^{-\frac{1}{n+1}}\ddc\psi_{2,j}\\
    \geq &-\omega_\varphi+r\omega_{\psi_1}+\epsilon_{2,j}\frac{n}{n+1}(-\psi_{2,j}+\Lambda_{2,j})^{-\frac{1}{n+1}}\omega_{\psi_{2,j}}.
\end{aligned}    \end{equation}
Taking trace with respect to the linearized operator $G^{i\bar{j}}:=\frac{\partial g}{\partial\lambda_k}\frac{\partial \lambda_k}{\partial h_{i\bar{j}}}$, we get \[
\begin{aligned}
 \frac{\partial g}{\partial\lambda_k}\frac{\partial\lambda_k}{\partial h_{i\bar{j}}}(\omega_M^{-1}\omega_{\varphi})_{i\bar{j}}=\frac{\partial g}{\partial \lambda_k}\lambda_{k}\geq & rn\left(\frac{\omega_{\psi_1}^n}{\omega_M^n}\right)^{\frac{1}{n}}(\det(G^{i\bar{j}}))^{\frac{1}{n}}\\
 &+\epsilon_{2,j}\frac{n^2}{n+1}(-\psi_{2,j}+\Lambda_{2,j})^{-\frac{1}{n+1}}\left(\frac{\omega_{\psi_{2,j}}^n}{\omega_M^n}\right)^{\frac{1}{n}}(\det G^{i\bar{j}})^{\frac{1}{n}}.
\end{aligned}
  \]
By our assumption on $ g $, we have $\det(G^{i\bar{j}})=\det(\frac{\partial g}{\partial h_{i\bar{j}}})\geq \gamma_0$.

Combine with $\sum_{k=1}^{n}\frac{\partial g}{\partial\lambda_k}\lambda_k\leq c_0 g$ and \eqref{geenralatmaximumineq}, we have \[
    c_0g\geq rn\left(\frac{\omega_{\psi_1}^n}{\omega_M^n}\right)^{\frac{1}{n}}\gamma_0^{\frac{1}{n}}+\epsilon_{2,j}\frac{n^2}{n+1}(-\psi_{2,j}+\Lambda_{2,j})^{-\frac{1}{n+1}}\left(\frac{V_\omega\eta_j(-\varphi+a\psi-s)f^n}{A_{s,j}}\right)^{\frac{1}{n}}\gamma_0^{\frac{1}{n}}.
\]
We choose $r=c_0an^{-1}\gamma_0^{-\frac{1}{n}}$, and $a$ small such that $r<1$, then by \eqref{generalsettingcomparison}, we get \[
c_0c_\omega f\geq \epsilon_{2,j}\frac{n^2}{n+1}(-\psi_{2,j}+\Lambda_{2,j})^{-\frac{1}{n+1}}\left(\frac{V_\omega\eta_j(-\varphi+a\psi-s)f^n}{A_{s,j}}\right)^{\frac{1}{n}}\gamma_0^{-1}.
\]
By our choice of $r,\epsilon_{2,j},\Lambda_{2,j}$, we get \[
    \eta_j(-\varphi+r\psi_1-s)\leq \epsilon_{2,j}(-\psi_{2,j}+\Lambda_{2,j})^{\frac{n}{n+1}}.
\]Note that $\eta_j(x)\geq x_+$, the proof of the claim $\Psi\leq 0$ is done. 

Next, by lemma 8 in \cite{MR4593734}, for any $\lambda[h_\varphi]\in \Gamma\subset\Gamma_1$ with $\sup\varphi=0$, we have $||\varphi||_{L^p(\omega_M^n)}\leq C_p$ for any $p\in [1,\frac{n}{n-1})$. Fix a $p\in(1,\frac{n}{n-1})$, note that $\psi_1\leq 0$, we can bound $A_s$ uniformly as \[
    A_{s}=\int_M(-\varphi+r\psi_1-s)_+f^n\omega_M^n\leq C||\varphi||_{L^p(\omega_M^n)}|f||_{L^{\frac{np}{p-1}}(\omega_M^n)}^{1-1/p}.
\]
Choosing $q$ at the beginning s.t. $q\geq \frac{np}{p-1}$, we get an uniform bound of $A_s$. As $A_{s,j}$ converges to $A_s$ as $j\rightarrow \infty$, $\Lambda_{2,j}$ is uniformly bounded. A similar iteration argument leads to \[
-\varphi+r\psi_1\leq C,
\]for a uniform constant $C.$ 
Note that we need only to show the theorem for $\beta$ small. 
So if we choose $\frac{\alpha}{q}\leq \beta$, we get the desired estimate by plug in the expression of $\psi$. 

\end{proof}

\begin{remark}
One can prove similar Moser-Trudinger type inequalities as in section 3 for a specific choice of $\Phi$.
\end{remark}

Next, following the paper of Guo-Phong \cite{guo2024uniform}, we prove a relative uniform estimate under the assumption of the existence of $\cC$-subsolution introduced by Sz\'ekelyhidi in \cite{MR3807322}.  Compared with Guo-Phong's result, we assume a weaker Orlicz norm of the volume form. Let's set up the equation first. 

Consider a compact K\"ahler manifold $(M,\omega_M)$ with a fixed background K\"ahler metric $\omega_M$ and $[\omega_M]^n=1$. Let $\omega$ be another K\"ahler metric and $\chi$ be a closed real $(1,1)$-form on $M$. Define $h_\varphi:=\omega^{-1}(\chi+\ddc\varphi)=\omega^{-1}\chi_\varphi$ as an endomorphism between $TM$ and itself. $\lambda[h_\varphi]$ is the vector of eigenvalues of $h_\varphi$. We change the notation a little bit since we allow $\omega$ to degenerate in some sense now. Define $F_\omega$, and the entropy $\cN_\Phi(F_\omega)$ as \[
    F_\omega=\log\left(\frac{\omega^n/[\omega]^n}{\omega_M^n/[\omega_M]^n}\right), \quad \cN_\Phi(\omega):=\int_M\Phi(F_\omega)e^{F_{\omega}}\omega_M^n.
\]
We assume that $\Phi(x):\mathbb{R}\rightarrow \mathbb{R}^+$ is an increasing function such that $\int^\infty\Phi^{-\frac{1}{n}}(t)dt<\infty$.

Let $g$ be a $C^1$-function defined on an open convex symmetric cone $\Gamma$ which satisfies $\Gamma_n\subset\Gamma\subset\Gamma_1$. We assume that $g$ is positive and symmetric in $\Gamma$, and $\frac{\partial g}{\partial\lambda_i}>0,\forall i=1,\ldots,n$.
The equation we consider now is \begin{equation}\label{generalequationsubsolution}
    g(\lambda[h_\varphi])=e^F,\quad \lambda[h_\varphi]\in \Gamma, \quad \sup \varphi=0.
\end{equation}

Following \cite{MR3807322} We say that $u$ is a $(\delta,R)$-subsolution to equation \eqref{generalequationsubsolution} if \[(\lambda[h_u]-\delta\mathds{1}+\Gamma_n)\cap\partial\Gamma^{F(x)}\subset B(0,R)\subset\mathbb{R}^n, \forall x\in M,
\]where $\mathds{1}=(1,\ldots,1)\in \mathbb{R}^n$, and $\partial\Gamma^{F(z)}:=\{\lambda\in \Gamma|g(\lambda)\geq F(z)\}$.

Suppose $u$ is a $(\delta,R)$-subsolution to equation \eqref{generalequationsubsolution}, and \[-n\kappa_1\leq tr_\omega(\chi+\ddc u)\leq \kappa_2.\]
\begin{theorem}
Let the assumptions be given as above and additionally $[\omega]\cdot[\omega_M]^{n-1}\leq A$. Let $\varphi$ be a smooth solution to equation \eqref{generalequationsubsolution}. There exists a constant $C$ depending on $M,n,\omega_M,\kappa_1,\kappa_2,A,\delta,R,\Phi,\cN_\Phi(F_\omega)$ such that \[
    \sup|(\varphi-u)-\sup(\varphi-u)|\leq C.
\]
\end{theorem}

\begin{proof}
By lemma \ref{choiceofrepresentative}, there exists a smooth real $(1,1)$-form $\theta\in [\omega]$ with $\omega=\theta+\ddc\varphi_\omega$ and $||\theta||_{C^k(\omega_M)}\leq C(A,\omega_M)$ for a uniform constant $C(A,\omega_M)$. We use $u_\beta$ to approximate the upper envelope of nonpositive $\theta$-psh functions $\cV_\theta$ constructed by lemma \ref{smoothapproxofupperenvelope}.

Consider the auxiliary equation \[
    (\theta+\psi_1)^{n}=[\omega]^n\frac{\Phi(F_\omega)e^{F_\omega}\omega_M^n}{\cN_\Phi(\omega)}, \quad \sup\psi_1=0.
\]
For any $a>0,q>1$, define $h$ as in \eqref{definitionofh}, and $\psi:=u_\beta-h(-\frac{\alpha}{q}(\psi_1-u_\beta-1)+\Lambda_1)$ .
Then we have the comparison inequality \begin{equation}\label{generalsubsolutioncomparison}
    \omega^n=[\omega]^ne^{F_\omega}\omega_M^n\leq a^n\theta_{\psi}^n+[\omega]^nf\omega_M^n,
\end{equation}with $f=-\frac{\alpha}{q}\psi_1+\Lambda_1$.
Now let $r=(R+\delta+\kappa_1)a$. Choose $a$ small enough to ensure $r\in(0,\frac{1}{2}\delta)$. We use $\varphi'$ to denote $\varphi-u-\sup(\varphi-u)$. 
Then consider another auxiliary equation \[
    (\theta+\ddc\psi_{2,j})^n=[\omega]^n\frac{\eta_j(-\varphi'-r\varphi_\omega+r\psi-s)f\omega_M^n}{A_{s,j}}, \quad \sup \psi_{2,j}=0,
\] where $A_{s,j}:=\int_M\eta_j(-\varphi'-\varphi_\omega+r\psi-s)f\omega_M^n.$

By theorem \ref{degenerateuniformestimate}, we have $0\leq -\varphi+\cV_\theta\leq B$ for a uniform constant $B$. Therefore $-\psi_{2,j}+\varphi_\omega+B=-\psi_{2,j}+\cV_\theta-(-\varphi_\omega+\cV_\theta)+B\geq 0.$

Choose constants $\epsilon_{2,j},\Lambda_{2,j}$ as \[{}
    \epsilon_{2,j}=(R+\delta+\kappa_1)^{\frac{1}{n+1}}\left(\frac{n+1}{n}\right)^{\frac{n}{n+1}}A_{s,j}^{\frac{1}{n+1}},\quad \Lambda_{2,j}=\frac{4}{\delta^{n+1}}(R+\delta+\kappa_1)\frac{n}{n+1}A_{s,j}.
\]
Consider the function \[
    \Psi:=-\varphi'-r\varphi_\omega+r\psi-s+\epsilon_{2,j}(-\psi_{2,j}+\varphi_\omega+B+\Lambda_{2,j})^{\frac{n}{n+1}}.
\]We claim that $\Psi\leq 0$. For simplicity, we use $\psi_{2,j}'$ to denote $\psi_{2,j}-\varphi_\omega-B$ in the following of the proof. Without loss of generality, we can assume that the maximum of $\Psi$ is attained at $x_0\in \Omega_{s}:=\{-\varphi'-r\varphi_\omega+r\psi-s\geq 0\}$. Then, at $x_0$, \[
    0\geq -\ddc\varphi'-r\ddc\varphi_\omega+r\ddc\psi+\epsilon_{2,j}\frac{n}{n+1}(-\psi_{2,j}'+\Lambda_{2,j})^{-\frac{1}{n+1}}\ddc\psi_{2,j}'
\]
Adding $\chi_u$ to both sides, and note that $r\leq \frac{\delta}{2}$ and $\epsilon_{2,j}\frac{n}{n+1}\Lambda_{2,j}^{-\frac{1}{n+1}}\leq \frac{\delta}{2}$, we have \[
    \chi_\varphi\geq \chi_u+r\theta_\psi+\epsilon_{2,j}\frac{n}{n+1}(-\psi_{2,j}+\Lambda_{2,j})^{-\frac{1}{n+1}}\theta_{\psi_{2,j}}-\delta\omega.
\]
Therefore, $\chi_\varphi=\chi_u-\delta\omega+\rho$ for some positive $(1,1)$-form $\rho$. By our assumption that $u$ is a $(\delta,R)$-subsolution, we get \[
    \lambda[h_\varphi]\in B(0,R)\subset\mathbb{R}^n.
\]
Hence,\[
    \lambda[\omega^{-1}(\chi_u+r\omega_\psi+\epsilon_{2,j}\frac{n}{n+1}(-\psi_{2,j}'+\Lambda_{2,j})^{-\frac{1}{n+1}}\theta_{\psi_{2,j}}-\delta\omega)]\in B(0,R),
\]i.e.\[
    nR\geq tr_\omega\chi_u+rtr_\omega\theta_\psi++\epsilon_{2,j}\frac{n}{n+1}(-\psi_{2,j}'+\Lambda_{2,j})^{-\frac{1}{n+1}}tr_\omega\theta_{\psi_{2,j}}-n\delta.
\]
Since we assume $tr_\omega\chi_u\geq -n\kappa_1$, and using the mean-value inequality we get \[
    nR+n\delta+n\kappa_1\geq r\left(\frac{\theta_\psi^n}{\omega^n}\right)^{\frac{1}{n}}+\epsilon_{2,j}\frac{n}{n+1}(-\psi_{2,j}'+\Lambda_{2,j})^{-\frac{1}{n+1}}\left(\frac{\theta_{\psi_{2,j}}^n}{\omega^n}\right)^{\frac{1}{n}}\]
Using inequality $(x+y)^n\geq x^n+y^n$ for any $x,y>0$, and combining with comparison inequality \eqref{generalsubsolutioncomparison} we get \[
    (R+\delta+\kappa_1)^n[\omega]^nf\omega_M\geq \epsilon_{2,j}^2\left(\frac{n}{n+1}\right)^{n}(-\psi_{2,j}'+\Lambda_{2,j})^{-\frac{n}{n+1}}\theta_{\psi_{2,j}}^n.
\]
By our choice of $\epsilon_{2,j},\Lambda_{2,j}$ and definition of $\theta_{\psi_{2,j}}$, the above inequality is equivalent to \[
    \eta_{j}(-\varphi'+r\psi-s)\leq \epsilon_{2,j}(-\psi_{2,j}'+\Lambda_{2,j})^{\frac{n}{n+1}}.
\]

Next, we show that $A_{s,j}$ is uniformly bounded from above and so is $\Lambda_{2,j}$. Since we assume $tr_\omega\chi_u\leq \kappa_2$ and $\lambda[\omega^{-1}(\chi_u+\ddc\varphi')]\in \Gamma\subset\Gamma_1$, we have $\Delta_\omega\varphi'\geq -\kappa_3$. By the theorem \ref{greenfunction}, we have \[
    0\leq \frac{1}{[\omega]^n}\int_Mv\omega^n+\int_M(G-\inf G)\kappa_3\omega^n.
\]
Thus $\int_M(-\varphi')f\omega_M^n\leq C$ for some uniform constant $C$. 

Now we start the iteration. Let $\phi(s):=\int_{\Omega_s}f\omega_M^n$, then \[
    \begin{aligned}
    A_s=&\int_{\Omega_s}(-\varphi'-r\varphi_\omega+r\psi-s)f\omega_M^n\\
    \leq &C(R,\delta,\kappa_1,n)A_{s,j}^{\frac{1}{n+1}}\int_M(-\psi_{2,j}'+\Lambda_{2,j})^{\frac{n}{n+1}}f\omega_M^n\\
    \leq &C(R,\delta,\kappa_1,n)A_{s,j}^{\frac{1}{n+1}}\left(\int_M(-\psi_{2,j}'+\Lambda_{2,j})^{\frac{pn}{n+1}}f\omega_M^n\right)^{\frac{1}{p}}\phi(s)^{1-\frac{1}{p}}\\
    \leq & C(R,\delta,\kappa_1,n,\omega_M,\kappa_2)A_{s,j}\phi(s)^{1-\frac{1}{p}}.\\
    \leq &C(R,\delta,\kappa_1,n)A_{s,j}^{\frac{1}{n+1}}(||-\psi_{2,j}'||_{L^\frac{pnq}{(n+1)(q-1)}(\omega_M^n)}^{\frac{n+1}{np}}+\Lambda_{2,j}^{\frac{np}{n+1}}[\omega_M]^n)||f||_{L^q(\omega_M^n)}\phi(s)^{1-\frac{1}{p}}\\
    \leq &C(R,\delta,\kappa_1,\kappa_2,n,q,p,\omega_M,M,A)A_{s,j}^{\frac{1}{n+1}}\phi(s)^{1-\frac{1}{p}}.
    \end{aligned}
\]
Fix $p>n+1$ and let $\delta:=(1-\frac{1}{p})\frac{n+1}{n}-1>0$, and  $j\rightarrow \infty$, we get \[
    A_s\leq C\phi(s)^{1+\delta},\quad \text{ for some uniform constant }C. 
\]
Together with $t\phi(s+t)\leq A_s$ and De Giorgi's lemma \ref{DeGiorgiiterationlemma} we get \[
    -\varphi'-r\varphi_\omega+r\psi\leq C(R,\delta,\kappa_1,\kappa_2,n,q,p,\omega_M,M,A).
\]
From the definition of $\psi=u_{\beta}+h(-\frac{\alpha}{q}\psi_1+\Lambda_1)$ and the fact $-\varphi_\omega+u_\beta\geq 0$, we get \[
    -\varphi'+h(-\frac{\alpha}{q}\psi_1+\Lambda_1)\leq C(R,\delta,\kappa_1,\kappa_2,n,q,p,\omega_M,M,A).
\]
Since we assume $\int^\infty\Phi^{-\frac{1}{n}}(t)dt<\infty$, we have $0\leq h(s)\leq C(\Phi,\cN(\omega))$, which closes the proof.
\end{proof}


\AtNextBibliography{\small}
\begingroup
\setlength\bibitemsep{2pt}
\printbibliography

\end{document}